\documentclass{qam-l}

\usepackage{amsfonts}
\usepackage{amsthm}

\usepackage{graphicx}
\usepackage{makecell,rotating}
\usepackage[labelsep=quad,indention=10pt]{caption}
\usepackage[list=true]{subcaption}
\usepackage{hyperref}

\DeclareCaptionSubType*[arabic]{table}
\DeclareCaptionSubType*[arabic]{figure}
\usepackage{chngcntr}

\usepackage{bigints}

\makeatletter
\newcommand*{\rom}[1]{\expandafter\@slowromancap\romannumeral #1@}
\makeatother
\usepackage{amsfonts}
\usepackage{amsthm,amssymb,amsmath}

\theoremstyle{plain}
\newtheorem{THEOREM}{Theorem}[section]

\newtheorem{theorem}[THEOREM]{Theorem}

\newtheorem{corollary}[THEOREM]{Corollary}
\newtheorem{lemma}[THEOREM]{Lemma}
\newtheorem{proposition}[THEOREM]{Proposition}

\theoremstyle{definition}

\theoremstyle{remark}

\newtheorem{remark}[THEOREM]{Remark}

\newcommand{\R}{\ensuremath{\mathbb{R}}}   

\def \a {\alpha}

\def \d {\delta}

\def \e {\varepsilon}

\def \l {\lambda}

\def \z {\zeta}

\def \O {\Omega}
\def \Rn {\R^n}

\begin{document}

\title{On the quenching behavior of the MEMS with fringing field}
\author{Xue Luo}
\address{School of Mathematics and Systems Science, Beihang University, Haidian District, Beijing, 100191, P. R. China.} %
\email{xluo@buaa.edu.cn, luoxue0327@163.com}
\author{Stephen S.-T. Yau}
\address{Department of Mathematical Sciences, Tsinghua University, Beijing, 100084, P.R.China.}
\email{yau@uic.edu}
\thanks{This work is supported by the start-up fund from Beihang Univeristy and Tsinghua University.}%

\subjclass[2000]{35J60, 35B40}

\date{}

\begin{abstract}
The singular parabolic problem $u_t-\triangle u=\l{\frac{1+\d|\nabla u|^2}{(1-u)^2}}$ on a bounded domain $\O$ of $\Rn$ with Dirichlet boundary condition, models the Microelectromechanical systems (MEMS) device with fringing field. In this paper, we focus on the quenching behavior of the solution to this equation. We first show that there exists a critical value $\l_\d^*>0$ such that if $0<\l<\l_\d^*$, all solutions exist globally; while for $\l>\l_\d^*$, all the solution will quench in finite time. The estimate of the quenching time in terms of large voltage $\l$ is investigated. Furthermore, the quenching set is a compact subset of $\O$, provided $\O$ is a convex bounded domain in $\mathbb{R}^n$. In particular, if the domain $\O$ is radially symmetric, then the origin is the only quenching point. We not only derive the one-side estimate of the quenching rate, but also further study the refined asymptotic behavior of the finite quenching solution.
\end{abstract}

\maketitle

\section{Introduction}

Micro- and nanoelectromechanical systems (MEMS and NEMS) are indubitably the hottest topic in engineering nowadays. These devices have been playing important roles in the development of many commercial systems, such as accelerometers, optical switches, microgrippers, micro force gauges, transducers, micropumps, etc. Yet it remains many researches to be done. A deeper understanding of basic phenomena will advance the design in MEMS and NEMS. 

The simplified physical model of MEMS is the idealized electrostatic device. The upper part of this device consists of a thin and deformable elastic membrane that is held fixed along its boundary and which lies above a rigid grounded plate. This elastic membrane is modeled as a dielectric with a small but finite thickness. The upper surface of the membrane is coated with a negligibly thin metallic conducting film. When a voltage $V$ is applied to the conducting film, the thin dielectric membrane deflects towards the bottom plate, and when $V$ is increased beyond a certain critical value $V^*$, which is known as {\sl pull-in voltage}, the steady-state of the elastic membrane is lost, and proceeds to quench or touch down at finite time.

In designing almost any MEMS or NEMS device based on the interaction of electrostatic forces with elastic structures, the designers will always confront the ``pull-in" instability. This instability refers to the pheonomena of quenching or touch down as we described previously when the applied voltage is beyond certain critical value $V^*$. It is easy to see that this instability severely restricts the stable range of operation of many devices \cite{PB}. Hence many reseaches have been done in understanding and controlling the instability.  Most investigations of MEMS and NEMS have followed Nathanson's lead \cite{NNW} and use some sort of small aspect ratio approximation to simplify the mathematical model. An overview of the physical phenomena of the mathematical models associated with the rapidly developing field of MEMS technology is given in \cite{PB}. 

The instability of the simplified mathematical model (cf. \cite{GPW}) has also been observed and analyzed in \cite{GPW}, \cite{GG2008}, \cite{Guo}, etc. This model is described by a partial differential equation:
\begin{equation}\label{eqn-without fringing}
	\left\{\begin{aligned}
		u_t-\triangle u=&\frac\l{(1-u)^2}\quad\textup{for}\ (x,t)\in\O_T\\
		u(x,t)=&0\quad x\in\partial\O_T\\
		u(x,0)=&0\quad x\in\O,
	\end{aligned}\right.
\end{equation} 
where $\circ_T=\circ\times[0,T)$, $T$ is the maximal time of existence of the solution. The study of \eqref{eqn-without fringing} starts from its stationary equation. It is shown in \cite{GG2007} that there exists a pull-in voltage $\l^*:=\l^*(\O)>0$ such that 
\begin{enumerate}
	\item[a.] If $0\leq \l<\l^*$, there exists at least one solution to the stationary equation of \eqref{eqn-without fringing}.
	\item[b.] If $\l>\l^*$, there is no solution to the stationary equation of \eqref{eqn-without fringing}.
\end{enumerate}
Concerning the evolutionary equation \eqref{eqn-without fringing}, \cite{GG2008} dealt with the issues of global convergence as well as finite and infinite time quenching of \eqref{eqn-without fringing}. It asserts that for the same $\l^*$ above, the followings hold:
\begin{enumerate}
	\item If $\l\leq\l^*$, then there exists a unique solution $u(x,t)$ to \eqref{eqn-without fringing} which globally converges pointwisely as $t\rightarrow+\infty$ to its unique minimal steady-state.
	\item If $\l>\l^*$, then a unique solution $u(x,t)$ to \eqref{eqn-without fringing} must quench in finite time.
\end{enumerate}
More refined analysis of the quenching behavior of \eqref{eqn-without fringing} is in \cite{GG2008}, \cite{Guo} and the references therein.

As pointed out in \cite{PD}, \eqref{eqn-without fringing} is only a leading-order outer approximation of an asymptotic theory based on expansion in the small aspect ratio. The fringing term $\d|\nabla u|^2$ is the first-order correction. The model \eqref{eqn-without fringing} is modified as:
\begin{equation}\tag{$F_{\l,\d}$}\label{F}
    \left\{ \begin{aligned}
            &u_t-\triangle u = \l\frac{1+\d|\nabla u|^2}{(1-u)^2},\quad (x,t)\in\O_T\\
            &u(x,t)=0, \quad (x,t)\in\partial\O_T\\
            &u(x,0)=0, \quad x\in\O.
        \end{aligned} \right.
\end{equation}
In this paper, we are aim to understand how the fringing term affects the behavior of the solution to \eqref{F}, including the pull-in voltage, quenching time, quenching behavior, etc.

The stationary equation of \eqref{F}
\begin{equation}\tag{$SF_{\l,\d}$}\label{SF}
    \left\{ \begin{aligned}
            &-\triangle u = \l\frac{1+\d|\nabla u|^2}{(1-u)^2},\quad x\in\O\subset\Rn\\
            &u(x)=0, \quad x\in\partial\O.
        \end{aligned} \right.
\end{equation}
has been studied in \cite{WY}. The authors show that for fixed $\d>0$, there exists a pull-in voltage $\l_\d^*>0$ such that for $\l>\l_\d^*$ there are no solution to \eqref{SF}; for $0<\l<\l_\d^*$ there are at least two solutions; and when $\l=\l_\d^*$ there exists a unique solution. Furthermore, for $\l<\l_\d^*$ the equation \eqref{SF} has a minimal solution $u_\l$ and $\l\mapsto u_\l$ is increasing for $\l\in(0,\l_\d^*)$.

The instability of \eqref{F} is stated in the following theorem. 

\begin{theorem}[Theorem 2.3, \cite{Wq}]\label{thm-intro.global and quenching}
    For fixed $\d>0$, suppose $\l_\d^*$ is the pull-in voltage in \cite{WY}, then the following hold:
\begin{enumerate}
    \item If $\l\leq\l_\d^*$, then there exits a unique global solution $u(x,t)$ of \eqref{F} which converges as $t\rightarrow\infty$ monotonically and pointwisely to its unique minimal steady state.
    \item If $\l>\l^*_\d$, then the unique solution $u(x,t)$ for \eqref{F} must quench in finite time.
\end{enumerate}
\end{theorem}
In the literature, we say the solution $u$ quenches if it reaches $u=1$. Although the proof of this theorem has been briefly sketched in \cite{Wq} with the right-hand side of \eqref{F} to be, rather than $\frac{1+\d|\nabla u|^2}{(1-u)^2}$, even more general nonlinearity $g(u)(1+\d|\nabla u|^2)$, where $g: [0,1)\rightarrow\mathbb{R}_+$ satisfying
\begin{center}
	$g$ is a $C^2$, positive, nondecreasing and convex function such that $\lim_{u\rightarrow1^-}g(u)=+\infty$, $\int_0^1g(s)ds=+\infty$.
\end{center} 
We believe the argument there is not rigorous, since when passing to the limit, it is not clear why $\lim_{t\rightarrow+\infty}\nabla k(x,t)=m(x)$ and $\lim_{t\rightarrow+\infty}\triangle k(x,t)=\triangle m(x)$. Instead, in this paper we adapt the argument in \cite{BCMR} to give a detailed proof.

The pull-in voltage $\l_\d^*$ has been estimated in \cite{Wq}:
\begin{align}\label{eqn-intro.lower bound for ld}
	\frac4{27}\frac{||\xi||_\infty}{||\xi||_\infty^2+\d||\triangle\xi||_\infty^2}\leq\l_\d^*\leq\l^*.
\end{align}
We show in this paper that 
\[
	\lim_{\d\rightarrow\infty}\l_\d^*=0.
\]
This improves the upper bound in \eqref{eqn-intro.lower bound for ld} dramatically for $\d\gg1$.

From Theorem \ref{thm-intro.global and quenching}, we know that the solution quenches in finite time when $\l\geq\l_\d^*$, denoted $T=T(\l,\d)<\infty$. The precise definition of quenching time $T$ is 
\[
	T=\sup\{t>0:\,||u(\cdot,\tau)||_\infty<1,\,\forall\tau\in[0,t]\}.
\]
It has been shown in \cite{Wq} that $T=\mathcal{O}\left((\l-\l^*_\d)^{-\frac12}\right)$, provided that $\l>\l^*_\d$ is sufficiently close to $\l_\d^*$ and $\d\ll1$. For $\l\gg\l^*_\d$, we show that the following result:

\begin{theorem}\label{thm-intro.upper bound for T}
	The quenching time $T=T(\l,\d)$ for the solution $u$ of \eqref{F} verifies
	\[
		\limsup_{\l\rightarrow\infty}\l T=\frac13.
	\]
\end{theorem}
This result is valid for \eqref{F} with or without fringing term. However, it is known that the quenching time for \eqref{F} without fringing term satisfies
\[
	\lim_{\l\rightarrow\infty}\l T=\frac13.
\]
The numerical results in section 6.1 suggest that $\lim_{\l\rightarrow\infty}\l T=0$. Acctually, with the similar argument in \cite{YZ}, we show that
\[
	T\leq\frac{||\phi||_1}{3\l||\phi||_1+||\triangle\phi||_1},
\]
where $\phi\geq0$ is any $C^2$ function in $\O$ and $\phi=0$ on $\partial\O$, and $||\cdot||_1$ is the $L^1$ norm. This implies that
\[
	T\lesssim\frac1\l,
\]
if $\l\gg1$. The notation $a\lesssim b$ means there exists some constant $C>0$ such that $a\leq Cb$. This is a finer decaying rate than $\mathcal{O}(\l^{-\frac12})$, which obtained in \cite{Wq}.
Besides the quenching time, we are also interested in the quenching set. The mathematical definition of quenching set is
\[
	\Sigma=\{x\in\bar{\O}:\,\exists\,(x_n,t_n)\in\O_T,\ s.t.\ x_n\rightarrow x,\ t_n\rightarrow T,\ u(x_n,t_n)\rightarrow1\}.
\]
We assume $\O\subset\mathbb{R}^n$ is a convex bounded domain. By the moving plane argument, we assert that the quenching set is a compact subset of $\O$. And if $\O=B_R$, the ball centered at the origin with the radius $R$, then the quenching solution is radially symmetric (cf. \cite{GNN}) and the only quenching point is the origin.
\begin{theorem}\label{quench at origin}
    Suppose $\O=B_R$. If $\l>\l_\d^*$, then the solution quenches only at $r=0$. That is, the origin is the unique quenching point.
\end{theorem}

To understand the quenching behavior of the finite time quenching solution to \eqref{F}, we begin with the one-side quenching esitmate, which has been derived in \cite{LW} for only one dimensional case.
\begin{lemma}[One-side quenching estimate]\label{lemma-upper bound}
	If $\O\subset\Rn$ is a convex bounded domain, and $u(x,t)$ is a quenching solution of \eqref{F} in finite time, then there exists a bounded positive constant $M>0$ such that 
\[
	M(T-t)^{\frac13}\leq1-u(x,t),
\]
for all $\O\times(0,T]$. Moreover, $u_t\rightarrow+\infty$ as $u$ touches down.
\end{lemma}

Acctually, we show in this paper that under certain condition (namely \eqref{intro.condition}), the solution quenches in finite time $T$ with the rate 
\[
	1-u(x,t)\sim(3\l(T-t))^{\frac13},
\]
as $t\rightarrow T^-$, provided $\O\in\mathbb{R}$ or $\O\in\mathbb{R}^n$, $n\geq2$, is radially symmetric domain. 

This result comes from the similarity variables, which is first suggested in \cite{GK}-\cite{GK1989}. Let us make the similarity transformation at some point $a\in\O_\eta$ as in \cite{GK} and \cite{Guo}:
\begin{align}\label{ysw}
	y=\frac{x-a}{\sqrt{T-t}},\quad s=-\log{(T-t)},\quad u(x,t)=1-(T-t)^{\frac13}w_a(y,s),
\end{align}
where $\O_\eta=\{x\in\O:\,dist(x,\partial\O)>\eta\}$, for some $\eta\ll1$. First, the point $a$ can be identified as a non-quenching point, if $w_a(y,s)\rightarrow\infty$, as $s\rightarrow+\infty$ uniformly in $|y|\leq C$, for any constant $C>0$. This is called the {\sl nondegeneracy} phenomena in \cite{GK1989}. This property is not difficult to derive. It follows immediately from the comparison principle and the nondegeneracy of \eqref{eqn-without fringing} obtained in \cite{Guo}.

The basis of the method, the similarity variables in \cite{Guo}, is the scaling property of \eqref{eqn-without fringing}, the fact that if $u$ solves it near $(0,0)$, then so do the rescaled functions 
\begin{equation}\label{u_gamma}
	1-u_\gamma(x,t)=\gamma^{-\frac23}\left[1-u(\gamma x,\gamma^2t)\right],
\end{equation}
for each $\gamma>0$. If $(0,0)$ is a quenching point, then the asymptotics of the quenching are encoded in the behavior of $u_\gamma$ as $\gamma\rightarrow0$. Unfortunately, compared with \eqref{eqn-without fringing}, \eqref{F} doesn't possess the nice property. That is, it is not rescale-invariant. This is where the difficulty in analysis arises and the condition \eqref{intro.condition} comes from. Essentially, we characterize the asymptotic behavior near a singularity, assuming a certain upper bound on the rate of the gradient's blow-up. The condition \eqref{intro.condition} in some degree forces the solution of \eqref{F} converges to the {\sl self-similar} solution of \eqref{eqn-without fringing} as $t\rightarrow T^-$. We call $u$ is the self-similar solution to \eqref{eqn-without fringing}, if $u$ defines on $\mathbb{R}^n\times(0,+\infty)$ and  $u_\gamma=u$ for every $\gamma$ (see \eqref{u_gamma}).

Hence, the study of the asymptotic behavior of $u$ near the singularity is equivalent to understand the behavior of $w_a(y,s)$, as $s\rightarrow+\infty$, which satisfies the equation:
\begin{equation}\label{eqn-intro.wa}
	\frac{\partial w_a}{\partial s}=\triangle w_a-\frac y2\cdot\nabla w_a+\frac13w_a-\frac\l{w_a^2}-\l\d e^{\frac s3}\frac{|\nabla w_a|^2}{w_a^2}.
\end{equation}

\begin{theorem}\label{thm-intro.asymptotics}
	Suppose $w_a$ is the solution to \eqref{eqn-intro.wa} quenching at $x=a$ in finite time $T$. Assume further that 
\begin{align}\label{intro.condition}
	\int_{s_0}^\infty se^{\frac s3}\int_{B_s}\rho|\nabla w_a|^2dyds<\infty,
\end{align}
for some $s_0\gg1$, where $\rho(y)=e^{-\frac{|y|^2}4}$, $B_s$ is defined in \eqref{eqn-Bs}. Then $w_a(y,s)\rightarrow w_{a\,\infty}(y)$, as $s\rightarrow\infty$ uniformly on $|y|\leq C$, where $C>0$ is any bounded constant, and $w_{a\,\infty}(y)$ is a bounded positive solution of
\begin{equation}\label{stationary w-eqn}
	\triangle w-\frac12y\cdot\nabla w+\frac13w-\frac\l{w^2}=0
\end{equation}
in $\mathbb{R}^n$. Moreover, if $\O\in\mathbb{R}$ or $\O\in\mathbb{R}^n$, $n\geq2$, is a convex bounded domain, then we have
	\[
		\lim_{t\rightarrow T^-}(1-u(x,t))(T-t)^{-\frac13}=(3\l)^{\frac13}
	\] 
uniformly on $|x-a|\leq C\sqrt{T-t}$ for any bounded constant $C$.
\end{theorem}

From Theorem \ref{thm-intro.asymptotics}, one hardly tells the effects of the fringing term $\d|\nabla u|^2$ on the asymptotic behavior near the singularity. Therefore, it seems to be necessary to find the refined asymptotic expansion near the singularity. As the first attempt in this direction, we derive a formal expansion as in \cite{KL} and \cite{GPW}. Let us consider $\O\subset\mathbb{R}^n$ be a radially symmetrical domain. Then, for $r\ll1$ and $T-t\ll1$, we have
\begin{equation}
	u\sim1-\left[3\l(T-t)\right]^{\frac13}\left(1-\frac{3^{\frac13}n}{8\d\l^{\frac23}}(T-t)^{\frac13}+\frac{3^{\frac13}}{4\d\l^{\frac23}}\frac{r^2}{(T-t)^{\frac23}}+\cdots\right)^{\frac13}.
\end{equation}
This expansion is quite different from the one for \cite{GPW}:
\[
	u\sim-1+[3\l(T-t)]^{\frac13}\left(1-\frac{3n}{4|\log{(T-t)}|}+\frac{3r^2}{8(T-t)|\log{(T-t)}|}+\cdots\right)^{\frac13}.
\]
We believe the difference is due to the fringing term, which can be clearly seen from the method of dominant balance, see detailed analysis in section 5.6. 

Finally, as the supplements, we numerically compute the pull-in voltages of \eqref{F} with various $\d$ and the quenching times of \eqref{F} with various $\d$ and $\l>\l_\d^*$ using {\it bvp4c} in Matlab. Furthermore, we solve \eqref{F} numerically using an appropriate finite difference scheme. The numerical simulations validate the results obtained in the previous sections.

\section{Global existence or quenching in finite time}

\setcounter{equation}{0}

Motivated by \cite{WY}, we make the following transformation
\begin{align}\label{transformation}
    v(x,t):=\zeta_{\l,\d}(u(x,t))=\int_0^{u(x,t)}e^{\frac{\l\d}{1-s}}ds,
\end{align}
then $v(x,t)$ satisfies
\begin{equation}\tag{$V_{\l,\d}$}\label{V}
    \left\{ \begin{aligned}
            &v_t-\triangle v = \l\rho_{\l,\d}(v),\quad (x,t)\in\O_T\\
            &v(x,t)=0, \quad (x,t)\in\partial\O_T\\
            &v(x,0)=0, \quad x\in\O,
        \end{aligned} \right.
\end{equation}
where $\rho_{\l,\d}=\xi_{\l,\d}\circ\zeta_{\l,\d}^{-1}(v)$, $\xi_{\l,\d}(u):=\frac{e^{\frac{\l\d}{1-u}}}{(1-u)^2}$. Since $\xi_{\l,\d}$ and $\zeta_{\l,\d}$ are increasing in $[0,1)$ and  $\lim_{u\rightarrow1^-}\zeta_{\l,\d}(u)=\infty$, $\rho_{\l,\d}$ is also increasing in $\R_+$. It is also not difficult to check that $\rho_{\l,\d}(v)$ satisfies the following properties:
\begin{enumerate}
    \item $\rho_{\l,\d}(v)$, $\rho_{\l,\d}'(v)$ and $\rho_{\l,\d}''(v)>0$, for $v\in\R_+$. In fact, through direct computations we get
\begin{align}\label{rho'}
    \rho_{\l,\d}'(v) = \frac1{(1-u)^3}\left(2+\frac{\l\d}{1-u}\right);\quad
    \rho_{\l,\d}''(v) = \frac2{(1-u)^4}e^{-\frac{\l\d}{1-u}}\left(3+\frac{2\l\d}{1-u}\right).
\end{align}
\item $\int_{v_0}^\infty\frac{ds}{\rho_{\l,\d}(s)}<\infty$, for any $v_0\in(0,\infty)$, since
\begin{align*}
    \int_{v_0}^\infty\frac{ds}{\rho_{\l,\d}(s)}&=\int_{\zeta_{\l,\d}^{-1}(v_0)}^1\frac{\zeta_{\l,\d}'(u)du}{\frac{e^{\frac{\l\d}{1-u}}}{(1-u)^2}}=\int_{\zeta_{\l,\d}^{-1}(v_0)}^1(1-u)^2du
    =\frac13\left(1-\zeta_{\l,\d}^{-1}(v_0)\right)^3<\infty.
\end{align*}
\end{enumerate}

\begin{lemma}[Uniqueness]
    Suppose $u_1(x,t)$ and $u_2(x,t)$ are solutions of \eqref{F} on $\O_T:=\O\times[0,T]$ such that $||u_i||_{L^\infty(\O_T)}<1$ for $i=1,2$, then $u_1=u_2$.
\end{lemma}
\begin{proof}
    Let us denote $v_i(x,t)$ the solutions of \eqref{V}, i.e. $v_i=\zeta_{\l,\d}(u_i)$, $i=1,2$. Then $\hat{v}=v_1-v_2$ satisfies
\begin{align}\label{hat_v}
    \hat{v}_t-\triangle\hat{v} = \l[\rho_{\l,\d}(v_1)-\rho_{\l,\d}(v_2)] = \l\frac{\rho_{\l,\d}(v_1)-\rho_{\l,\d}(v_2)}{\hat{v}}\hat{v}:=\l f\hat{v}.
\end{align}
The condition $||u_i||_{L^\infty(\O_T)}<1$ is equivalent to $||v_i||_{L^\infty(\O_T)}<\infty$, $i=1,2$. This implies that $||\rho_{\l,\d}(v_i)||_{L^\infty(\O_T)}<\infty$, $i=1,2$. Therefore, $||f||_{L^\infty{(\O_T)}}<\infty$.

We now fix $T_1\in(0,T)$ and consider the solution $\phi$ of the problem
\begin{equation}
    \left\{ \begin{aligned}
            &\phi_t+\triangle\phi+\l f\phi=0,\quad (x,t)\in\O_{T_1}\\
            &\phi(x,T_1)=\theta(x)\in C_0(\O), \quad x\in\O\\
        &\phi(x,t)=0,\quad (x,t)\in\partial\O_{T_1}.
        \end{aligned} \right.
\end{equation}
The standard linear theory gives the unique and bounded solution (cf. Theorem 8.1, \cite{LSU}).

Multiplying $\phi$ to \eqref{hat_v} and integrating in $\O_{T_1}$ on both sides, it yields by integration by parts that
\begin{align*}
    \int_\O \hat{v}(x,T_1)\theta(x)dx=0,
\end{align*}
for arbitrary $T_1\in(0,T)$ and $\theta(x)\in C_0(\O)$. This implies that $\hat{v}\equiv0$.
\end{proof}

\subsection{Global existence}

\begin{theorem}[Global existence]\label{thm-global existence}
    For every $\l\leq\l_\d^*$, there exists a unique global solution $u(x,t)$ of \eqref{F}, which monotonically converges as $t\rightarrow\infty$ to the minimal solution $u_\l$ of \eqref{SF}.
\end{theorem}
\begin{proof}
    This is standard and follows from the maximum principle combined with the existence of the regular minimal steady-state solutions for $\l\in(0,\l_\d^*)$. Indeed, for any $0<\l\leq\l_\d^*$, from Theorem 1 and Theorem 5, \cite{WY}, there exits a unique minimal solution $u_\l(x)$ of \eqref{SF}. It is clear that $0$ and $u_\l$ are sub- and super-solutions to \eqref{F}, respectively. This implies that there exists a unique global solution $u(x,t)$ of \eqref{F} such that $1>u_\l(x)\geq u(x,t)\geq0$ in $\O\times(0,\infty)$. Let us denote $v_\l=\zeta_{\l,\d}(u_\l)<\infty$. Then, $0\leq v(x,t)\leq v_\l<\infty$, where $v=\zeta_{\l,\d}(u)$.

By differentiating \eqref{V} in time and setting $w=v_t$, we get for any fixed $t_0>0$
\begin{equation*}
    \left\{\begin{aligned}
        w_t-\triangle w= \left[\frac{\l\d}{(1-u)^4}+\frac2{(1-u)^3}\right]w,&\quad (x,t)\in\O_{t_0}\\
        w(x,t)=0,&\quad (x,t)\in\partial\O\times(0,t_0)\\
        w(x,0)\geq0,&\quad x\in\O.
    \end{aligned}\right.
\end{equation*}
Here $\left[\frac{\l\d}{(1-u)^4}+\frac2{(1-u)^3}\right]$ is a locally bounded non-negative function, and by the strong maximum principle, we get that $v_t=w>0$ for $\O_{t_0}$ or $w=0$. The second case can't happen, otherwise $u(x,t)=u_\l(x)$ for any $t>0$. It follows that $w=v_t>0$ for all $\O\times(0,\infty)$. Moreover, since $v(x,t)$ is bounded, the mononicity in time implies that the unique solution $v(x,t)$ converges to some steady state, denoted as $v_{ss}(x)$, as $t\rightarrow\infty$, i.e. $u(x,t)\rightarrow u_{ss}(x)$, as $t\rightarrow\infty$. Hence, $1>u_\l(x)\geq u_{ss}(x)>0$ in $\O$.

Next, we claim that $u_{ss}(x)$ is a solution of \eqref{SF}. Let us consider $v_1(x)$ satisfying
\begin{equation*}
    \left\{\begin{aligned}
        -\triangle v_1&=\l\rho_{\l,\d}(v_{ss}),\quad x\in\O\\
        v_1(x)&=0,\quad x\in\partial\O.
    \end{aligned}\right.
\end{equation*}
Let $\bar{v}(x,t)=v(x,t)-v_1(x)$, then $\bar{v}$ satisfies $\bar{v}(x,0)=-v_1(x)$, $\bar{v}|_{\partial\O\times(0,\infty)}=0$ and
\begin{align}\label{w}
    \bar{v}_t-\triangle\bar{v}=\l\left[\frac{e^{\frac{\l\d}{1-u}}}{(1-u)^2}-\frac{e^{\frac{\l\d}{1-u_{ss}}}}{(1-u_{ss})^2}\right],
\end{align}
in $\O$. The right-hand side of \eqref{w} tends to zero in $L^2(\O)$, as $t\rightarrow\infty$, which follows from
\begin{align*}
    \left|\frac{e^{\frac{\l\d}{1-u}}}{(1-u)^2}-\frac{e^{\frac{\l\d}{1-u_{ss}}}}{(1-u_{ss})^2}\right|
    \overset{\eqref{rho'}}\leq e^{\frac{\l\d}{1-u_\l}}\left[\frac{\l\d}{(1-u_\l)^4}+\frac2{(1-u_\l)^3}\right]|u-u_{ss}|,
\end{align*}
and H\"older's inequality. A standard eigenfunction expansion implies that $w(x,t)$ converges to zero in $L^2(\O)$ as $t\rightarrow\infty$. That is $v(x,t)\rightarrow v_1(x)$, as $t\rightarrow\infty$. Combined with the fact that $v(x,t)\rightarrow v_{ss}(x)$ pointwisely as $t\rightarrow\infty$. We deduce that $v_1(x)=v_{ss}(x)$ in $L^2(\O)$, which implies $v_{ss}(x)$ is also a solution to the stationary equation of \eqref{V} and the corresponding $u_{ss}(x)$ is also a solution to \eqref{SF}. The minimal property of $u_\l$ yields that $u_\l\equiv u_s$ in $\O$, from which follows that for every $x\in\O$, we have $u(x,t)\uparrow u_\l(x)$, as $t\rightarrow\infty$.
\end{proof}

\subsection{Finite-time quenching}

\begin{theorem}[Finite-time quenching]\label{thm-finite time quenching}
    For every $\l>\l_\d^*$, there exits a finite time $T=T(\l,\d)$ at which the unique solution $u(x,t)$ of \eqref{F} quenches.
\end{theorem}
\begin{proof}
    By contradiction, let $\l>\l_\d^*$ and suppose there exists a solution $u(x,t)$ of \eqref{F} in $\O\times(0,\infty)$.

Claim: given any $\e\in(0,\l-\l_\d^*)$, $(F_{\l-\e,\d})$ has a global solution $u_\e$, which is uniformly bounded in $\O\times(0,\infty)$ by some constant $C_\e<1$.

We follow the similar argument  as in \cite{BCMR} or \cite{GG}. Let
\begin{align*}
    g(u)&=\frac1{(1-u)^2}, \quad h(u)=\int_0^u\frac{ds}{g(s)},\quad0\leq u\leq1;\\
    \tilde{g}(u)&=\frac{\l-\e}{\l(1-u)^2},\quad \tilde{h}(u)=\int_0^u\frac{ds}{\tilde{g}(s)},\quad0\leq u\leq1;
\end{align*}
and
\begin{align*}
    \Phi_\e(u)=\tilde{h}^{-1}\circ h(u).
\end{align*}
Direct computations yield that
\begin{align*}
    \Phi_\e(u)=1-\left[\frac\e\l+\frac{\l-\e}{\l}(1-u)^3\right]^{\frac13}\leq C_\e<1,
\end{align*}
for $0\leq u\leq1$, where $C_\e=1-\left(\frac\e\l\right)^{\frac13}$. Moreover, it is easy to check that $\Phi_\e(0)=0$, $0\leq\Phi_\e(s)<s$, for $s\geq0$, and $\Phi_\e(s)$ is increasing and concave with
\begin{align*}
    \Phi_\e'(s)=\frac{\tilde{g}\circ\Phi_\e(s)}{g(s)}>0.
\end{align*}
Setting $w_\e=\Phi_\e(u)$, we have
\begin{align*}
    -\triangle w_\e&=-\Phi_\e''(u)|\nabla u|^2-\Phi_\e'(u)\triangle u
\geq \Phi_\e'(u)\left[\l\frac{(1+\d|\nabla u|^2)}{(1-u)^2}-u_t\right]\\
&=\l(1+\d|\nabla u|^2)\frac{\Phi_\e'(u)}{(1-u)^2}-(w_\e)_t=\l(1+\d|\nabla u|^2)\tilde{g}(w_\e)-(w_\e)_t\\
    &=\frac{(\l-\e)(1+\d|\nabla u|^2)}{(1-w_\e)^2}-(w_\e)_t.
\end{align*}
Notice that
\begin{align*}
    \Phi_\e'(u)=\frac{\frac{\l-\e}\l(1-u)^2}{\left[\frac\e\l+\frac{\l-\e}\l(1-u)^3\right]^{\frac23}}
    \leq\frac{\frac{\l-\e}\l(1-u)^2}{\left[\frac{\l-\e}\l(1-u)^3\right]^{\frac23}}
    =\left(\frac{\l-\e}\l\right)^{\frac13}<1.
\end{align*}
Hence,
\begin{align*}
    |\nabla w_\e|^2=(\Phi_\e'(u))^2|\nabla u|^2\leq|\nabla u|^2.
\end{align*}
Furthermore,
\begin{align*}
    -\triangle w_\e\geq\frac{(\l-\e)(1+\d|\nabla w_\e|^2)}{(1-w_\e)^2}-(w_\e)_t.
\end{align*}
This means that $w_\e=\Phi_\e(u)\leq C_\e$ is the supersolution to $(F_{\l-\e,\d})$. Since zero is a subsolution of $(F_{\l-\e,\d})$, we deduce that there exists a unique global solution $u_\e$ for $(F_{\l-\e,\d})$ satisfies $0\leq u_\e\leq w_\e\leq C_\e<1$ uniformly in $\O\times(0,\infty)$.

Let $v_\epsilon=\zeta_{\l-\epsilon,\d}(u_\epsilon)$. It is clear to see that $v_\epsilon$ is a global classical solution to $(V_{\l-\epsilon,\d})$. And it has been checked previously that $\rho_{\l-\epsilon,\d}$ is a nondecreasing, convex function, and there exists some $v_0>0$ such that $\rho_{\l-\epsilon,\d}(v_0)>0$ and
\begin{align*}
    \int_{v_0}^\infty\frac{ds}{\rho_{\bar{\l},\d}(s)}<\infty.
\end{align*}
Therefore, from Theorem 1, \cite{BCMR}, we obtain a weak solution to the stationary equation of ($V_{\l-\epsilon,\d}$), where $0<\e<\l-\l_\d^*$. In fact, using Sobolev embedding theorem and a boot-strap arguement, any weak solution to the stationary equation of ($V_{\l-\epsilon,\d}$) satisfying $\rho_{\l-\epsilon,\d}(v)\in L^1(\O)$ is indeed smooth. This contradicts with the nonexistence result in \cite{WY}.
\end{proof}

\section{Estimates for the pull-in voltage and the finite quenching time}

\setcounter{equation}{0}

A lower bound of $\l_\d^*$ is given in Theorem 2.2, \cite{Wq}, i.e.,
\begin{equation}\label{eqn-lower bound for ld}
	\l_l:=\frac4{27}\frac{||\xi||_\infty}{||\xi||_\infty^2+\d||\triangle\xi||_\infty^2}\leq\l_\d^*,
\end{equation}
where $\xi$ is the solution to $-\triangle\xi=1$, $x\in\O$ with the Dirichlet boundary condition. And it is not difficult to see that $\l^*$, the pull-in voltage for \eqref{eqn-without fringing}, is an upper bound for $\l_\d^*$, due to the comparison principle. From \cite{GPW}, $\l^*\leq\frac4{27}\mu_0$, where $\mu_0>0$ is the first eigenvalue of $-\triangle\phi_0=\mu_0\phi_0$, $x\in\O$ with Dirichlet boundary condition. We shall derive an upper bound for $\l_\d^*$ to show explicit dependence of $\d$, if $\d\ll1$:
\begin{proposition}[Upper bound for $\l_\d^*$, $\d\ll1$]\label{prop-lower bound for ld}
	The pull-in voltage $\l_\d^*<\infty$ of \eqref{F} has the upper bound
\[
	\l_\d^*\leq\l_{u,1}:=\frac4{27}\mu_0\left(1-\frac1{27||\xi||_\infty}\d+\mathcal{O}(\d^2)\right),
\]
where $\xi$ is the solution to $-\triangle\xi=1$, $x\in\O$ with Dirichlet boundary condition, if $\d\ll1$.
\end{proposition}
\begin{proof}
	This argument is used in many estimates of the pull-in voltage (cf. Theorem 3.1, \cite{Pe} or Theorem 2.1, \cite{GPW}). Let $\mu_0>0$ and $\phi_0>0$ be the first eigen pair $-\triangle\phi_0=\mu_0\phi_0$ in $\O$ with Dirichlet boundary condition. We multiply the stationary equation of \eqref{V} by $\phi_0$, integrate the resulting equation over $\O$, and use Green's identity to get
\[
	\int_\O [-\mu_0 v+\l\rho_{\l,\d}(v)]\phi_0dx=\bigintss_\O\left(-\mu_0\int_0^u e^{\frac{\l\d}{1-s}}ds+\l\frac{e^{\frac{\l\d}{1-u}}}{(1-u)^2}\right)dx=0.
\] 
Noting that $\int_0^ue^{\frac{\l\d}{1-s}}ds\leq e^{\frac{\l\d}{1-u}}u$, we get that if $\l>\frac4{27}\mu_0$, then 
\begin{equation}\label{eqn-estimate ld}
	-\mu_0\int_0^u e^{\frac{\l\d}{1-s}}ds+\l\frac{e^{\frac{\l\d}{1-u}}}{(1-u)^2}\geq-\mu_0e^{\frac{\l\d}{1-u}u}+\l\frac{e^{\frac{\l\d}{1-u}}}{(1-u)^2}>0,
\end{equation}
for any $u\in[0,1]$. Therefore, there is no solution to the stationary equation of \eqref{V}, so does \eqref{SF}, if $\l>\frac4{27}\mu_0$. That is, $\l_\d^*\leq\frac4{27}\mu_0$. This is the upper bound obtained in \cite{GPW}. In this way, we ignore the effect $\d$ completely. Let us go back to \eqref{eqn-estimate ld} and we see that if 
\[
	\l\geq\max_{u\in[0,1]}\left[\mu_0(1-u)^2\int_0^ue^{\l_l\d\left(\frac1{1-s}-\frac1{1-u}\right)}ds\right],
\]
then \eqref{eqn-estimate ld} holds, where $\l_l$ is in \eqref{eqn-lower bound for ld}. Let us estimate the maximum in the following:
\[
	\int_0^ue^{\l_l\d\left(\frac1{1-s}-\frac1{1-u}\right)}ds\leq\frac12u\left[1+e^{\l_l\d\left(1-\frac1{1-u}\right)}\right],
\]
due to the convexity of the integrand $e^{\l_l\d\left(\frac1{1-s}-\frac1{1-u}\right)}$. Therefore, if 
\[
	\l\geq\max_{u\in[0,1]}\frac12\mu_0(1-u)^2u\left[1+e^{\l_l\d\left(1-\frac1{1-u}\right)}\right]=\frac4{27}\mu_0-\frac{4\mu_0}{27^2||\xi||_\infty}\d+\mathcal{O}(\d^2),
\]
then \eqref{eqn-estimate ld} holds, where $\xi$ is the solution to $-\triangle\xi=1$ for $x\in\O$ with Dirichlet boundary condition, provided $\d\ll1$. 
\end{proof}

Next, we show the behavior of $\l_\d^*$ as $\d\rightarrow\infty$.
\begin{proposition}[$\l_\d^*$ for $\d\gg1$]\label{prop-delta large}
	The pull-in voltage $\l_\d^*$ of \eqref{F} tends to $0$, as $\d\rightarrow\infty$. That is,
\[
	\lim_{\d\rightarrow\infty}\l_\d^*=0.
\]
\end{proposition} 
\begin{proof}
	As shown in Theorem \ref{thm-global existence} and Theorem \ref{thm-finite time quenching}, the pull-in voltage $\l_\d^*$ of \eqref{F} is the same one as that of \eqref{SF}. Let us multiply \eqref{SF} by $\phi_0>0$, the first eigenfunction of $-\triangle\phi_0=\mu_0\phi_0$ in $\O$ with Dirichlet boundary condition, integrate over $\O$, and use Green's identity to get
\begin{equation}\label{eqn-lambda large}
	0=\int_\O\left(-\mu_0u+\frac\l{(1-u)^2}\right)\phi_0dx+\l\d\int_\O\frac{|\nabla u|^2}{(1-u)^2}\phi_0dx.
\end{equation}
By integration by parts, the third term in the above equation gets
\begin{align}\label{eqn-lambda large.3}\notag
	\int_\O\frac{|\nabla u|^2}{(1-u)^2}\phi_0dx=&\int_\O\nabla u\cdot\frac{\nabla u}{(1-u)^2}\phi_0dx
	=\int_\O\nabla u\cdot\nabla\left(\frac1{1-u}\right)\phi_0dx\\\notag
		=&-\int_\O\triangle u\frac1{1-u}\phi_0dx-\int_\O\nabla u\cdot\nabla\phi_0\frac1{1-u}dx\\\notag
		=&\l\int_\O\frac{1+\d|\nabla u|^2}{(1-u)^3}\phi_0dx+\int_\O\nabla(\ln{(1-u)})\cdot\nabla\phi_0dx\\
		=&\l\int_\O\frac{1+\d|\nabla u|^2}{(1-u)^3}\phi_0dx+\mu_0\int_\O\ln{(1-u)}\phi_0dx.
\end{align}
Furthermore, for $p\geq3$, we have
\begin{align}\label{eqn-lambda large.4}\notag
	\int_\O\frac{|\nabla u|^2}{(1-u)^p}\phi_0dx=&\frac1{p-1}\int_\O\nabla u\nabla\left(\frac1{(1-u)^{p-1}}\right)\phi_0dx\\\notag
	=&-\frac1{p-1}\int_\O\triangle u\frac1{(1-u)^{p-1}}\phi_0dx-\frac1{p-1}\int_\O\nabla u\nabla\phi_0\frac1{(1-u)^{p-1}}dx\\\notag
	=&\frac\l{p-1}\int_\O\frac{1+\d|\nabla u|^2}{(1-u)^{p+1}}\phi_0dx\\\notag
	&-\frac1{(p-1)(p-2)}\int_\O\nabla\phi_0\nabla\left(\frac1{(1-u)^{p-2}}\right)dx\\\notag
	=&\frac\l{p-1}\int_\O\frac{\phi_0}{(1-u)^{p+1}}dx+\frac{\l\d}{p-1}\int_\O\frac{|\nabla u|^2}{(1-u)^{p+1}}\phi_0dx\\\notag
	&-\frac{\mu_0}{(p-1)(p-2)}\int_\O\frac{\phi_0}{(1-u)^{p-2}}dx\\
	&-\frac1{(p-1)(p-2)}\int_{\partial\O}\frac{\partial\phi_0}{\partial\nu}\frac1{(1-u)^{p-2}}dS,
\end{align}
where $\nu$ is the outward unit normal vector of $\partial\O$. Substitute \eqref{eqn-lambda large.3} and \eqref{eqn-lambda large.4} to \eqref{eqn-lambda large}, we get
\begin{align}\label{eqn-lambda large.=}\notag
	0=&\bigintss_\O\left\{-\mu_0u+\frac\l{(1-u)^2}+\l\d\left[\mu_0\ln{(1-u)}+\frac\l{(1-u)^3}\right]\right\}\phi_0dx\\\notag
	&\phantom{\bigintss_\O}+\sum_{p=3}^P\frac{\l^{p-1}\d^{p-1}}{(p-1)!}\left[\l\int_\O\frac1{(1-u)^{p+1}}\phi_0dx-\frac{\mu_0}{p-2}\int_\O\frac1{(1-u)^{p-2}}\phi_0dx\right.\\\notag
	&\phantom{\bigintss_\O+\sum_{p=3}^P\frac{\l^{p-1}\d^{p-1}}{p-1}}\left.-\frac1{p-2}\int_{\partial\O}\frac{\partial\phi_0}{\partial\nu}\frac1{(1-u)^{p-2}}dS\right]\\
	&\phantom{\bigintss_\O}+\frac{\l^P\d^P}{(P-1)!}\int_\O\frac{|\nabla u|^2}{(1-u)^{P+1}}\phi_0dx,
\end{align}
for arbitrary $P>p$. By the boundary point lemma, we have $\frac{\partial\phi_0}{\partial\nu}<0$ on $\partial\O$. Hence, the term $\displaystyle{-\int_{\partial\O}\frac{\partial\phi_0}{\partial\nu}\frac1{(1-u)^{p-2}}dS}$ is positive, so does the term $\displaystyle{\int_\O\frac{|\nabla u|^2}{(1-u)^{P+1}}dx}$. If $\d\gg1$, then $\mathcal{O}(\d^{P-1})$ is the leading order term, except the last term in \eqref{eqn-lambda large.=}. The equality \eqref{eqn-lambda large.=} can't hold when
\[
	\frac\l{(1-u)^{P+1}}-\frac{\mu_0}{(P-2)(1-u)^{P-2}}>0
\] 
holds for all $u\in[0,1]$. That is,
\begin{align}\label{eqn-lambda large.0}
	\l>\frac{\mu_0}{P-2}\max_{u\in[0,1]}(1-u)^3=\frac{\mu_0}{P-2},
\end{align}
where $P=P(\d)\rightarrow\infty$, if $\d\rightarrow\infty$. Our result follows immediately.
\end{proof}
\begin{proposition}[Upper bound of $T$]\label{prop-upper bound large lambda}
	Let $\phi$ be any nonnegative $C^2$ function such that $\phi\not\equiv0$ and $\phi=0$ on $\partial\O$. Then for $\l$ large enough, the quenching time $T$ for the solution to \eqref{F} satisfies
	\begin{equation}\label{eqn-upper bound of T}
		T\leq\frac{||\phi||_1}{3\l||\phi||_1-||\triangle\phi||_1},
	\end{equation}
where $||\cdot||_1$ is the $L^1$ norm of $\cdot$ in $\O$.
\end{proposition}

\begin{proof}
	Using $\varphi(1-u)^2$ as the test function to \eqref{F} and integrating over $\O$,
\begin{align*}
	\left(\int_\O\frac13\left[1-(1-u)^3\right]\varphi dx\right)_t
	=&\int_\O\triangle u\varphi(1-u)^2dx+\l\int_\O(1+\d|\nabla u|^2)\varphi dx\\
	\geq&-\int_\O\nabla u\nabla\varphi(1-u)^2dx+2\int_\O|\nabla u|^2\varphi(1-u)\\
	&+\l\int_\O\varphi dx\\
	\geq&\frac13\int_\O\left[1-(1-u)^3\right]\triangle\varphi dx+\l\int_\O\varphi dx.
\end{align*}
Hence, for any $t<T$, integrating from $0$ to $t$, we obtain that
\begin{align*}
	\frac13\int_\O\varphi dx\geq&\frac13\int_\O\left[1-(1-u(x,t))^3\right]\varphi dx\\
\geq&\frac13\int_0^t\int_\O\left[1-(1-u)^3\right]\triangle\varphi dx+\l t\int_\O\varphi dx
\geq\l t\int_\O\varphi dx-\frac13t\int_\O|\triangle\varphi|dx.
\end{align*}
By tending $t$ to $T$, we are done.
\end{proof}

We compare the quenching time $T=T(\l,\d)$ with different $\l$:
\begin{proposition}\label{prop-different lambda}
    Suppose $u_1=u_1(x,t)$ and $u_2=u_2(x,t)$ are solutions of \eqref{F} with $\l=\l_1$ and $\l_2$, respectively. And the corresponding finite quenching times are $T_{\l_1}$ and $T_{\l_2}$, respectively. If $\l_1>\l_2$, then $T_{\l_1}<T_{\l_2}$.
\end{proposition}
\begin{proof}
    Let $\hat{v}=v_1-v_2$, where $v_i$, $i=1,2$, are the corresponding solution of $(V_{\l_i,\d})$, $i=1,2$, respectively. Then $\hat{v}|_{\partial\O}(x,t)=\hat{v}(x,0)=0$ and
\begin{align*}
    \hat{v}_t-\triangle\hat{v}=\l_1\rho_{\l_1,\d}(v_1)-\l_2\rho_{\l_2,\d}(v_2)
    >\l_2[\rho_{\l_2,\d}(v_1)-\rho_{\l_2,\d}(v_2)]=\l_2\rho'_{\l_2,\d}(\theta)\hat{v},
\end{align*}
with $\rho'_{\l_2,\d}(\theta)\geq0$, for some function $\theta$. Hence, $v_1>v_2$ in $\O\times(0,\min{\{T_{\l_1},T_{\l_2}\}})$. Thus, $T_{\l_1}<T_{\l_2}$.
\end{proof}

\begin{remark}\label{remark-different delta}
    Fix the voltage $\l>\max\{\l_{\d_1}^*,\l_{\d_2}^*\}$, if $\d_1>\d_2>0$, then $T_{\d_1}<T_{\d_2}$, where $T_{\d_i}$ are the finite quenching time corresponding to $\d_i$, $i=1,2$. This observation follows immediately from
\begin{align*}
    \partial_t u_1-\triangle u_1=\l\frac{1+\d_1|\nabla u_1|^2}{(1-u_1)^2}>\l\frac{1+\d_2|\nabla u_1|^2}{(1-u_1)^2},
\end{align*}
which means that $u_1>u_2$ in $\O_{\min\{T_{\d_1},T_{\d_2}\}}$. Hence, $T_{\d_1}<T_{\d_2}$.
\end{remark}

\section{Quenching set}

\setcounter{equation}{0}

In this section, we assume that $\O$ is a bounded convex subset of $\Rn$.
It is followed by the moving-plane argument that the quenching set of any finite-time quenching solution to \eqref{F} is a compact subset of $\O$.

\begin{theorem}[Compactness of the quenching set]\label{compact}
    Suppose $\O\subset\Rn$ is convex, and $u(x,t)$ is a solution to \eqref{F} which quenches in finite time $T$. Then the set of the quenching points is a compact subset of $\O$.
\end{theorem}
\begin{proof}(Adaption of moving-plane argument)
    It is equivalent to show that the set of the blow-up points of $v$ in \eqref{V} is a compact subset of $\O$.

Let us denote $x=(x_1,x')\in\Rn$, where $x'=(x_2,x_3,\cdots,x_n)\in\mathbb{R}^{n-1}$. Take any point $y_0\in\partial\O$, and assume without loss of generality that $y_0=0$ and that the half space $\{x_1>0\}$ is tangent to $\O$ at $y_0$.

Let $\O_\a^+=\O\cap\{x_1>\a\}$, $\a<0$, $|\a|$ small, and $\O_\a^-=\{x=(x_1,x')\in\Rn:\,(2\a-x_1,x')\in\O_\a^+\}$, the reflection of $\O_\a^+$ with respect to $\{x_1=\a\}$.

First, from the maximum principle, we observe that
\begin{align}\label{v>0}
    v\geq0,
\end{align}
for $(x,t)\in\O_T$ and $\frac{\partial v}{\partial\nu}(t_0)<0$ on $\partial\O$ for some small $t_0\in(0,T)$.

Let us consider
\[
    \bar{v}(x,t)=v(x_1,x',t)-v(2\a-x_1,x',t),
\]
for $x\in\O_\a^-$, then $\bar{v}$ satisfies
\[
    \partial_t\bar{v}-\triangle\bar{v}=\l\left[\rho_{\l,\d}(v(x_1,x',t))-\rho_{\l,\d}(v(2\a-x_1,x',t))\right]=\l c(x,t)\bar{v},
\]
where $c(x,t)$ is a bounded function. It is clear that $\bar{v}=0$ on $\{x_1=\a\}$ and $\bar{v}=v(x_1,x',t)\geq0$ on $\partial\O_\a^-\cap\{x_1<\a\}\times(0,T]$. If $\a$ is small enough, then $\bar{v}(x,t_0)=v(x_1,x',t_0)-v(2\a-x_1,x',t_0)\geq0$, for $x\in\O_\a^-$. Applying maximum principle, we conclude that
\[
    \bar{v}>0\quad\textup{in}\ \O_\a^-\times(t_0,T)\quad\textup{and}\quad\frac{\partial\bar{v}}{\partial v_1}=-\frac{2\partial v}{\partial x_1}>0\quad\textup{on}\ \{x_1=\a\}.
\]
Since $\a$ is arbitrary, it follows by varying $\a$ that
\begin{align}\label{v-bdry}
    \frac{\partial v}{\partial x_1}<0,
\end{align}
for $x\in\O_{\a_0}^+$, $t_0<t<T$, provided that $\a_0$ is small enough.

Let us consider
\[
    J=v_{x_1}+\e_1(x_1-\a_0)
\]
in $\O_{\a_0}^+\times(t_0,T)$, where $\e_1=\e_1(\a_0,t_0)>0$ is a constant to be determined later. Through direct computations, we obtain that
\begin{align*}
    \partial_t J-\triangle J=&\triangle(v_{x_1})+\l\frac{e^{\frac{\l\d}{1-u}}}{(1-u)^3}u_{x_1}\left[2+\frac{\l\d}{1-u}\right]-\triangle(v_{x_1})\\
    =&\frac{\l v_{x_1}}{(1-u)^3}\left[2+\frac{\l\d}{1-u}\right]\leq0,
\end{align*}
in $\O_{\a_0}^+\times(t_0,T)$, where $u$ is the corresponding solution to \eqref{F}. Therefore, $J$ can't obtain positive maximum in $\O_{\a_0}^+\times(t_0,T)$. Next, $J<0$ on $\{x_1=\a_0\}$ by \eqref{v-bdry}. From \eqref{v>0}, $\frac{\partial v(x,t_0)}{\partial x_1}\leq C<0$.

If we can show $J<0$ on $\Gamma\times(t_0,T)$, where $\Gamma=\partial\O_{\a_0}^+\cap\partial\O$, then
\begin{align}\label{J<0}
    J<0,
\end{align}
in $\O_{\a_0}^+\times(t_0,T)$. To show \eqref{J<0}, we compare $v$ with the solution $z$ of the heat equation
\begin{equation}\label{heat}
    \left\{\begin{aligned}
        z_t-\triangle z&=0\quad\textup{in}\quad\O\times(t_0,T)\\
        z(x,t)&=0\quad\textup{on}\quad\partial\O\times(t_0,T)\\
        z(x,t_0)&=0\quad\textup{in}\quad\O\times(t_0,T).
    \end{aligned}\right.
\end{equation}
Since  $\l\rho_{\l,\d}(v)\geq0$, we have $v\geq z$. Consequently, $\frac{\partial v}{\partial\nu}<\frac{\partial z}{\partial\nu}\leq-C_0<0$ on $\partial\O\times(t_0,T)$. It follows that, if $x\in\Gamma$,
\[
    J\leq-C_0\frac{\partial x_1}{\partial\nu}+\e_1(x_1-\a_0)<0,
\]
provided $\e_1$ small enough. Now, the maximum principle yields that there exists $\e_1=\e_1(\a_0,t_0)$ small enough such that $J\leq0$ in $\O_{\a_0}^+\times(t_0,T)$, i.e.
\[
    -v_{x_1}=|v_{x_1}|\geq\e_1(x_1-\a_0),
\]
if $x'=0$, $\a_0\leq x_1<0$. Integrating with respect to $x_1$, we get for any $\a_0<y_1<0$,
\[
    -v(y_1,0,t)+v(\a_0,0,t)\geq\frac{\e_1}2|y_1-\a_0|^2.
\]
It follows that
\begin{align*}
    \liminf_{t\rightarrow T^-}v(0,t)=&\liminf_{t\rightarrow T^-}\lim_{y_1\rightarrow0^-}v(y_1,0,t)\\
    \leq&\liminf_{t\rightarrow T^-}\lim_{y_1\rightarrow0^-}\left[v(\a_0,0,t)-\frac{\e_1}2|y_1-\a_0|^2\right]<\infty.
\end{align*}
Thus, every point in $\{x'=0,\ \a_0<x_1<0\}$ is not a blow-up point. The above proof shows that $\a_0$ can be chosen independent of $y_0\in\partial\O$. Hence, by varying $y_0\in\partial\O$, we conclude that there is an $\O$-neighborhood $\O'$ of $\partial\O$ such that each point $x\in\O'$ is not a blow-up point. Since the blow-up points lie in a compact subset of $\O$, it is clearly a closed set.
\end{proof}

In addition, if $\O=B_R(0)$ is a ball of radius $R$ centered at the
origin, then according to \cite{GNN} we conclude that any solution $u(x,t)$ is indeed
radial symmetric, i.e. $u(x,t)=u(r,t)$, with $r=|x|\in[0,R]$.
Furthermore, we can show that the only possible quenching point is
the origin.

\begin{theorem}\label{quench at origin}
    Suppose $\O=B_R$. If $\l>\l_\d^*$, then the solution quenches only at $r=0$. That is, the origin is the unique quenching point.
\end{theorem}
\begin{lemma}\label{lemma-v_r negative}
    $v_r<0$ in $\O_T\cap\{r>0\}$.
\end{lemma}
\begin{proof}
    Set $\bar{v}=r^{n-1}v_r$. Then \eqref{V} becomes
\begin{align}\label{vt-wr}
    v_t-\frac1{r^{n-1}}\bar{v}_r=\l\rho_{\l,\d}(v).
\end{align}
Differentiating with respect to $r$, we get
\[
    \bar{v}_t+\frac{n-1}r\bar{v}_r-\bar{v}_{rr}-\frac\l{(1-u)^3}\left(2+\frac{\l\d}{1-u}\right)\bar{v}=0.
\]
since $\bar{v}=r^{n-1}v_r<0$ on $\partial\O\times(0,T)$ (by maximum principle) and $\bar{v}(r,0)=0$ by \eqref{V}, we deduce by maximum principle that $v_r<0$ in $\O_T\cap\{r>0\}$.
\end{proof}

\begin{proof}[Proof of Theorem \ref{quench at origin}]
Let us consider as in Theorem 2.3, \cite{FM}
\[
    J=\bar{v}+c(r)F(v),
\]
where $\bar{v}=r^{n-1}v_r$ is defined as in Lemma \ref{lemma-v_r negative}, $F$, $c$ are positive functions to be determined and
$F'\geq0$, $F''\geq0$. We aim to show $J\leq0$ in $\O_T$.
Through direct computations, we have
\begin{align*}
    J_t+\frac{n-1}rJ_r-J_{rr}
    =&\frac\l{(1-u)^3}\left[2+\frac{\l\d}{1-u}\right]\bar{v}\\
        &+cF'f+\frac{2(n-1)}rcF'v_r+\frac{n-1}rc'F-cF''v_r^2-2c'F'v_r-c''F\\
    \leq&\left\{\frac\l{(1-u)^3}\left[2+\frac{\l\d}{1-u}\right]+\frac{2(n-1)}{r^n}cF'-\frac{2c'F'}{r^{n-1}}\right\}J\\
        &+\left\{-\frac\l{(1-u)^3}\left[2+\frac{\l\d}{1-u}\right]cF+cF'\frac{\l e^{\frac{\l\d}{1-u}}}{(1-u)^2}\right.\\
        &\phantom{+\{\}}\left.-\frac{2(n-1)}{r^n}c^2F'F+\frac{n-1}rc'F+\frac2{r^{n-1}}cc'F'F-c''F\right\}\\
    :=& AJ+B,
\end{align*}
by using $\bar{v}=J-cF$ and $\bar{v}=r^{n-1}v_r$. It is easy to see that $A$
is a bounded function for $0<r<R$. Let us choose
\begin{align*}
    c(r)=\e r^n\quad\textup{and}\quad F(v)=\frac{e^{\frac{\l\d}{1-u}}}{(1-u)^\gamma},
\end{align*}
where $u=\zeta_{\l,\d}^{-1}(v)$, $\gamma\geq0$ is some constant to be determined later. Direct computations yield that
\begin{align*}
    B=c(r)e^{\frac{\l\d}{1-u}}\left\{\frac{\l(\gamma-2)}{(1-u)^{\gamma+3}}+2\e\left[\frac{\l\d}{(1-u)^{2\gamma+2}}+\frac\gamma{(1-u)^{2\gamma+1}}\right]\right\}\leq0,
\end{align*}
if $\gamma<1$ and $\e\ll1$. $J=0$ at $r=0$, due
to $c(0)=0$ and it follows that $J$ can't obtain positive maximum in
$\O_T$ or on $\{t=T\}$.

Next, we observe that $J$ can't obtain positive maximum on
$\{r=R\}$, if $J_r\leq0$ on $\{r=R\}$. Since
\[
    J_r(R)=\bar{v}_r+cF'v_r+c'F\leq \bar{v}_r+c'F\overset{\eqref{vt-wr}}=-R^{n-1}\l
    e^{\l\d}+c'(R)F(0)\leq0,
\]
provided that $\e\ll1$. Finally, by maximum
principle, there exists $0<t_0<T$ such that $v_r(r,t_0)<0$ for
$0<r\leq R$ and $v_{rr}(0,t_0)<0$. Thus, $J(r,t_0)<0$ for $0\leq
r<R$, provided $\e\ll1$.

Therefore, by maximum principle, we conclude that $J\leq0$ in
$B_R\times[t_0,T]$, for any $0<t_0<T$. That is,
\[
    -r^{n-1}e^{\frac{\l\d}{1-u}}u_r=-r^{n-1}v_r\geq c(r)F(v)=\frac{\e r^ne^{\frac{\l\d}{1-u}}}{(1-u)^\gamma},
\]
for $0\leq\gamma<1$. It deduces that 
\[
	\frac{d}{dr}\left[\frac1{\gamma+1}[1-u(r,t)]^{\gamma+1}\right]\geq\e r.
\]
Integrating from $0$ to $r$, we obtain that
\[
	\frac1{\gamma+1}[1-u(r,t)]^{\gamma+1}-\frac1{\gamma+1}[1-u(0,t)]^{\gamma+1}\geq\frac12\e r^2.
\]
It is known that $0$ is in the set of quenching points. So,
\begin{align}\label{1-u}
	[1-u(r,t)]^{\gamma+1}\geq\frac{\gamma+1}2\e r^2.
\end{align}
If for any $0<r<R$, $u(r,t)\rightarrow1$, as $t\rightarrow T$, then the left-hand side tends to $0$. This contradicts with \eqref{1-u}. Therefore, $0$ is the only quenching point.
\end{proof}

\section{Quenching behavior}

\setcounter{equation}{0}

\subsection{Upper bound estimate}

We first obtain an one-side quenching estimate. The similar result has been obtained in \cite{LW} for only one dimension case, i.e., $x\in\mathbb{R}$.

\begin{lemma}[One-side quenching estimate]\label{lemma-upper bound}
	If $\O\subset\Rn$ is a bounded convex domain, and $u(x,t)$ is a quenching solution of \eqref{F} in finite time, then there exists a bounded positive constant $M>0$ such that 
\[
	M(T-t)^{\frac13}\leq1-u(x,t),
\]
for all $\O_T$. Moreover, $u_t\rightarrow+\infty$ as $u$ quenches.
\end{lemma}
\begin{proof}
	Since $\O$ is a convex bounded domain, we show in Theorem \ref{compact} that the quenching set of $u$ is a compact subset of $\O$. It is now suffices to discuss the point $x_0$ lying in the interior domain $\O_\eta=\{x\in\O:\,\textup{dist}(x,\partial\O)>\eta\}$, for some small $\eta>0$, i.e. there is no quenching point in $\O_\eta^c:=\O\setminus\O_\eta$.

For any $t_1<T$, we recall the maximum principle gives $u_t>0$, for all $(x,t)\in\O\times(0, t_1)$. Furthermore, the boundary point lemma shows that the exterior normal derivative of $u_t$ on $\partial\O$ is negative for $t>0$. This implies that for any small $0<t_0<T$, there exists a positive constant $C=C(t_0,\eta)$ such that $u_t(x,t_0)\geq C>0$, for all $x\in\bar{\O}_\eta$. For any $0<t_0<t_1<T$, we claim that
\[
	J^\e(x,t)=v_t-\e \rho_{\l,\d}(v)\geq0,
\]
for all $(x,t)\in\O_\eta\times(t_0,t_1)$, where $v$ is the corresponding solution to \eqref{V}. In fact, it is clear that there exists $C_\eta=C(t_0,t_1,\eta)>0$ such that $v_t=e^{\frac{\l\d}{1-u}}u_t\geq C_\eta$ on $\O_\eta\times(t_0,t_1)$. And further, we can choose $\e=\e(t_0,t_1,\eta)>0$ small enough, so that $J^\e\geq0$ on the parabolic boundary of $\O_\eta\times(t_0,t_1)$, due to the local boundedness of $\rho_{\l,\d}(v)$ on $\partial\O_\eta\times(t_0,t_1)$. Then the claim is followed by the maximum principle and the direct computations:
\begin{align*}
	J_t^\e-\triangle J^\e=\l\rho_{\l,\d}'(v)J^\e+\e\rho_{\l,\d}''(v)|\nabla v|^2\geq\l\rho_{\l,\d}'(v)J^\e,
\end{align*} 
due to the convexity of $\rho_{\l,\d}$. This yields that for any $0<t_0<t_1<T$, there exists $\e=\e(t_0,t_1,\eta)>0$ such that 
\[
	u_t\geq\frac\e{(1-u)^2}, 
\] 
for all $\O_\eta\times(t_0,t_1)$. This inequality implies that $u_t\rightarrow\infty$ as $u$ touches down, and there exists $M>0$ such that
\begin{align}\label{one-side}
	M(T-t)^{\frac13}\leq1-u(x,t),
\end{align}
in $\O_\eta\times(0,T)$, due to the arbitrary of $t_0$ and $t_1$, where $M=M(\l,\d,\eta)$. Furthermore, one can obtain \eqref{one-side} for $\O\times(0,T]$, due to the boundedness of $u$ on $\O_\eta^c$.
\end{proof}

\subsection{Gradient esitmate}

We shall study the quenching rate for the higher derivatives of $u$. The idea of the proof is similar to Proposition 1, \cite{GK} and Lemma 2.6, \cite{Guo}.
\begin{lemma}\label{lemma-gradient estimate}
	Suppose $u$ is a quenching solution of \eqref{F} in finite time $T$. For any point $x=a\in\O_\eta$, for some small $\eta>0$. Then there exists a positive constant $M'$ such that
\begin{align}\label{gradient}
	|\nabla^mu(x,t)|(T-t)^{-\frac13+\frac m2}\leq M',
\end{align}
$m=1,2$, holds for $Q_R=B_R\times(T-R^2,T)$, for any $R>0$ such that $a+R\in\O_\eta$.
\end{lemma}
\begin{proof}
	It suffices to consider the case $a=0$ by translation. We may focus on some fixed $r$, such that $\frac12R^2<r^2<R^2$ and denote $Q_r=B_r\times\left(T\left(1-\left(\frac rR\right)^2\right), T\right)$. 

Let us first show that $|\nabla u|$ and $|\nabla^2u|$ are uniformly bounded on compact subset of $Q_R$. Indeed, since $\rho_{\l,\d}(v)$ is bounded on any compact subset $D$ of $Q_R$, standard $L^p$ estimates for heat equations (see \cite{LSU}) give
\[
	\iint_{D}\left(|\nabla^2 v|^p+|v_t|^p\right)dxdt<C,
\]
for $1<p<\infty$ and any cylinder $D$ with $\bar{D}\subset Q_R$. And it also holds for $u$, i.e.
\[
	\iint_{D}\left(|\nabla^2 u|^p+|u_t|^p\right)dxdt<C,
\]
$1<p<\infty$, where $C$ is a generic constant and may vary from line to line. Choosing $p$ large, by Sobolev embedding theorem, we conclude that $u$ is H\"older continuous on $D$, so does $\rho_{\l,\d}(v)$. Therefore, Schauder's estimates for heat equation (see \cite{LSU}) show that $|\nabla v|$ and $|\nabla^2v|$ are bounded on any compact subets of $D$, so do $|\nabla u|$ and $|\nabla^2 u|$. In particular, there exists $M_1$ such that
\[
	|\nabla u|+|\nabla^2 u|\leq M_1,
\]
for $(x,t)\in B_r\times\left(T\left(1-\left(\frac rR\right)^2\right), T\left(1-\frac12\left(1-\frac rR\right)^2\right)\right)$, where $M_1$ depends on $R$, $n$ and $M$ given in \eqref{one-side}. 

We next prove \eqref{gradient} for $B_r\times\left[T\left(1-\frac12\left(1-\frac rR\right)^2\right), T\right)$. For fixed point $(x,t)\in B_r\times\left[T\left(1-\frac12\left(1-\frac rR\right)^2\right), T\right)$, we consider
\begin{align}\label{w-transform}
	\bar{u}(z,\tau)=1-\mu^{-\frac23}\left[1-u\left(x+\mu z,T-\mu^2(T-\tau)\right)\right],
\end{align}
where $\mu=\left[2\left(1-\frac tT\right)\right]^\frac12$, which satisfies
\begin{equation}\label{w-eqn}
    \left\{ \begin{aligned}
            &\bar{u}_{\tau}-\triangle_z\bar{u} = \l\frac{1+\d\mu^{-\frac23}|\nabla_z \bar{u}|^2}{(1-\bar{u})^2},\quad (z,\tau)\in\mathcal{O}_T\\
            &\bar{u}(z,\tau)=1-\mu^{-\frac23}<0, \quad (z,\tau)\in\partial\mathcal{O}_T\\
            &\bar{u}(z,0)=\bar{u}_0(z), \quad z\in\mathcal{O},
        \end{aligned} \right.
\end{equation}
where $\bar{u}_0(z)=1-\mu^{-\frac23}[1-u(x+\mu z,T(1-\mu^2))]<0$ and $\triangle_z\bar{u}_0+\l\frac{1+\d\mu^{-\frac23}|\nabla_z\bar{u}_0|^2}{(1-\bar{u}_0)^2}>0$ on $\mathcal{O}$. For the fixed point $(x,t)$, we define $\mathcal{O}:=\{z:\,x+\mu z\in\O\}$. It is implied by \eqref{w-transform} that $T$ is also the finite quenching time of $\bar{u}$, and the domain of $\bar{u}$ includes $Q_{r_0}$ for some $r_0=r_0(R)>0$. Since the quenching set of $u$ is a compact subset of $\O$, due to Theorem \ref{compact}, so does that of $\bar{u}$. Therefore, the argument of Lemma \ref{lemma-upper bound} can be applied to \eqref{w-eqn}, yielding that there exists a constant $M_2>0$ such that
\[
	1-\bar{u}(z,\tau)\geq M_2(T-\tau)^{\frac13},
\]
where $M_2$ depends on $R$, $\l$, $\d$ and $\O$. Applying the interior $L^p$ estimates and Schauder's estimates to $\bar{u}$ as before, there exists $M'_1=M'_1(R,\l,\d,n,M_2)>0$ such that
\begin{align}\label{w-grad estimate}
	|\nabla_z \bar{u}|+|\nabla^2_z \bar{u}|\leq M'_1, 
\end{align}
for $(z,\tau)\in B_r\times\left(T\left(1-\left(\frac r{r_0}\right)^2\right),T\left(1-\frac12\left(1-\frac r{r_0}\right)^2\right)\right)$, where we assume that $\frac12r_0^2<r^2<r_0^2$. Applying \eqref{w-transform} and taking $(z,\tau)=\left(0,\frac T2\right)$, \eqref{w-grad estimate} gives 
\[
	\mu^{-\frac13+1}|\nabla_xu|+\mu^{-\frac13+2}|\nabla^2_xu|\leq M'_1.
\]
Thus, \eqref{gradient} follows immediately from $\mu=\left[2\left(1-\frac tT\right)\right]^{\frac12}$.
\end{proof}

\subsection{Lower bound estimate}

First, we note the following local lower bound estimate.
\begin{proposition}\label{prop-lower estimate}
	Suppose $u(x,t)$ is a quenching solution of \eqref{F} in finite time $T$. Then, there exists a bounded constant $C=C(\l,\O)>0$ such that
\begin{align}\label{upper bound}
	\max_{x\in\O}u(x,t)\geq1-C(T-t)^{\frac13},
\end{align}
for $0<t<T$.
\end{proposition}
\begin{proof}
	Let $U(t)=\max_{x\in\O}u(x,t)$, $0<t<T$, and let $U(t_i)=u(x_i,t_i)$, $i=1,2$ with $h=t_2-t_1>0$. Then,
\begin{align*}
	U(t_2)-U(t_1)&\geq u(x_1,t_2)-u(x_1,t_1)=hu_t(x_1,t_1)+o(1);\\
	U(t_2)-U(t_1)&\leq u(x_2,t_2)-u(x_2,t_1)=hu_t(x_2,t_2)+o(1).
\end{align*}
It follows that $U(t)$ is Lipschitz continuous. Hence, for $t_2>t_1$, we have
\[
	\frac{U(t_2)-U(t_1)}{t_2-t_1}\leq u_t(x_2,t_2)+o(1).
\]
On the other hand, since $\nabla u(x_2,t_2)=0$ and $\triangle u(x_2,t_2)\leq0$, we obtain
\[
	u_t(x_2,t_2)\leq\frac\l{(1-u(x_2,t_2))^2}=\frac\l{(1-U(t_2))^2},
\]
for $0<t_2<T$. Consequently, at any differentiable point of $U(t)$, it deduces from the above inequalities that
\begin{align}\label{U}
	(1-U)^2U_t\leq\l,
\end{align}
for a.e. $0<t<T$. \eqref{upper bound} is obtained by integrating \eqref{U} from $t$ to $T$.
\end{proof}

\subsection{Nondegeneracy of quenching solution}

For the quenching solution $u(x,t)$ of \eqref{F} in finite time $T$, we now introduce the associated similarity variables
\begin{align}\label{ysw}
	y=\frac{x-a}{\sqrt{T-t}},\quad s=-\log{(T-t)},\quad u(x,t)=1-(T-t)^{\frac13}w_a(y,s),
\end{align}
$a$ is any point in $\O_\eta$, for some small $\eta>0$. The form of $w_a$ defined in \eqref{ysw} is motivated by Lemma \ref{lemma-upper bound} and Proposition \ref{prop-lower estimate}. Then $w_a(y,s)$ is defined in 
\[
	W_a:=\{(y,s):\,a+ye^{-\frac s2}\in\O,\ s>s'=-\log{T}\},
\]
and it solves 
\begin{equation}\label{w_a-eqn}
	\frac{\partial w_a}{\partial s}=\triangle w_a-\frac y2\cdot\nabla w_a+\frac13w_a-\frac\l{w_a^2}-\l\d e^{\frac s3}\frac{|\nabla w_a|^2}{w_a^2}.
\end{equation}
Here $w_a$ is always strictly positive in $W_a$. The slice of $W_a$ at a given time $s=s_0$ will be denoted as $\O_a(s_0)$:
\[
	\O_a(s_0):=W_a\cap\{s=s_0\}=e^{\frac{s_0}2}(\O-a).
\]
For any $a\in \O_\eta$, there exists $s_0=s_0(\eta,a)>0$ such that
\begin{equation}\label{eqn-Bs}
	B_s:=\{y:\,|y|<s\}\subset\O_a(s),
\end{equation}
for $s\geq s_0$. 

Equation \eqref{w_a-eqn} could also be written in divergence form:
\begin{align}\label{w-divergence form}
	\rho w_s=\nabla(\rho\cdot\nabla w)+\frac13\rho w-\frac{\l\rho}{w^2}-\l\d\rho e^{\frac s3}\frac{|\nabla w|^2}{w^2},
\end{align}
with $\rho(y)=e^{-\frac{|y|^2}4}$.

We shall reach the nondegeneracy of the quenching behavior. The conclusion is obtained by the comparison principle \cite{F} and results in \cite{Guo}.
\begin{theorem}\label{thm-nondengeneracy}
	Suppose $u$ is a quenching solution of \eqref{F} in finite time $T$ and $a$ is any point in $\O_\eta$, for some $\eta>0$. If $w_a(y,s)\rightarrow\infty$ as $s\rightarrow\infty$ uniformly for $|y|\leq C$, where $C$ is any positive constant, then $a$ is not a quenching point of $u$.
\end{theorem}
\begin{proof}
	It is easy to see that $w_a$ in \eqref{w_a-eqn} is a subsolution of 
\[
	\frac{\partial}{\partial s}\tilde{w}=\triangle\tilde{w}-\frac y2\cdot\nabla\tilde{w}+\frac13\tilde{w}-\frac\l{\tilde{w}^2}
\]
in $B_{s_0}\times(s_0,\infty)$. From the comparison principle (cf. \cite{F}), we get $w_a\leq\tilde{w}$ in $B_{s_0}\times(s_0,\infty)$. If $w_a(y,s)\rightarrow\infty$, as $s\rightarrow\infty$ uniformly in $|y|\leq C$, so does $\tilde{w}(y,s)$. Our conclusion follows immediately from Theorem 2.12, \cite{Guo}, where $f\equiv1$ and $\tilde{w}$ is the $w_a$ in \cite{Guo}.
\end{proof}
\begin{remark}
	The proof of Theorem \ref{thm-nondengeneracy} also implies that the quenching set of the solution to ($F_{\l,0}$) is a subset of that of $u$, the solution to \eqref{F}, $\d>0$. 
\end{remark}

\subsection{Asymptotics of quenching solution} In this subsection, we shall omit all the subscription $a$ of $w_a$, $W_a$ and $\O_a$ if no confusion will arise.

In view of \eqref{ysw}, one combine Lemma \ref{lemma-upper bound} and Lemma \ref{lemma-gradient estimate} to reach the following estimates on $w$, $\nabla w$ and $\triangle w$:
\begin{corollary}\label{coro-estimate on w}
	Suppose $u$ is a quenching solution to \eqref{F} in finite time $T$. Then the rescaled solution $w$ satisfies
\[
	M\leq w\leq e^{\frac s3},\quad |\nabla w|+|\triangle w|\leq M',\quad\textup{in}\ W, 
\]
where $M$ and $M'$ are constants in Lemma \ref{lemma-upper bound} and Lemma \ref{lemma-gradient estimate}, respectively. Moreover, it satisfies 
\begin{equation}\label{estimate on w}	
M\leq w(y_1,s)\leq w(y_2,s)+M'|y_1-y_2|,
\end{equation}
for any $(y_i,s)\in W$, $i=1,2$.
\end{corollary}

\begin{lemma}\label{lemma-w_infty}
	Let $s_j$ be an increasing sequence such that $s_j\rightarrow+\infty$, and $w(y, s+s_j)$ is uniformly convergent to a limit $w_\infty(y,s)$ in compact sets. Then either $w_\infty(y,s)\equiv\infty$ or $w_\infty(y,s)<\infty$ in $\mathbb{R}^n$.
\end{lemma}
\begin{proof}
Inequality \eqref{estimate on w} implies that 
\[
	w_\infty(y_1,s)\leq w_\infty(y_2,s)+M'|y_1-y_2|
\]
and the conclusion follows.
\end{proof}

\begin{proposition}\label{prop-asymptotics}
	Suppose $w$ is the solution of \eqref{w_a-eqn} quenching at $x=a$ in finite time $T$. Assume further that 
\begin{align}\label{condition}
	\int_{s_0}^\infty se^{\frac s3}\int_{B_s}\rho|\nabla w|^2dyds<\infty,
\end{align}
for some $s_0\gg1$, where $\rho(y)=e^{-\frac{|y|^2}4}$, $B_s$ is defined in \eqref{eqn-Bs}. Then $w(y,s)\rightarrow w_\infty(y)$, as $s\rightarrow\infty$ uniformly on $|y|\leq C$, where $C>0$ is any bounded constant, and $w_\infty(y)$ is a bounded positive solution of
\begin{equation}\label{stationary w-eqn}
	\triangle w-\frac12y\cdot\nabla w+\frac13w-\frac\l{w^2}=0
\end{equation}
in $\mathbb{R}^n$.
\end{proposition}
\begin{proof}
	Let us adapt the arguments in the proofs of Proposition 6 and 7 \cite{GK} or Lemma 3.1 \cite{Guo}. Let $\{s_j\}$ be an increasing sequence tending to $\infty$ and $s_{j+1}-s_j\rightarrow\infty$. Let us denote $w_j(y,s)=w(y,s+s_j)$. Applying Arzela-Ascoli theorem on $z_j(y,s)=\frac1{w_j(y,s)}$ with Corollary \ref{coro-estimate on w}, there is a subsequence of $\{z_j\}$, still denoted as $\{z_j\}$, such that
\[
	z_j(y,s)\rightarrow z_\infty(y,s)
\]
uniformly on compact sets of $W$ and
\[
	\nabla z_j(y,m)\rightarrow \nabla z_\infty(y,m)
\]
for almost all $y$ and for each integer $m$. That is, $w_j(y,s)\rightarrow w_\infty(y,s)$ uniformly on the compact sets of $W$ and $\nabla w_j(y,m)\rightarrow\nabla w_\infty(y,m)$ for almost all $y$ and for each integer $m$. From Lemma \ref{lemma-w_infty}, we get that either $w_\infty(y,s)\equiv\infty$ or $w_\infty(y,s)<\infty$ in $y\in\mathbb{R}^n$. The case that $w_\infty(y,s)\equiv\infty$ could be excluded by Theorem \ref{thm-nondengeneracy}, since $a$ is the quenching point.

Let us define the associate energy of $w$ at time $s$:
\[
	E[w](s)=\frac12\int_{B_s}\rho|\nabla w|^2dy-\frac16\int_{B_s}\rho w^2dy-\l\int_{B_s}\frac\rho wdy.
\]
Direct computations yield that
\begin{align}\label{dE/ds}\notag
	\frac d{ds}E[w](s)=&\int_{B_s}\rho\nabla w\cdot\nabla w_sdy-\frac13\int_{B_s}\rho ww_sdy+\l\int_{B_s}\frac\rho{w^2}w_sdy\\\notag
		&+\frac1{2}\int_{\partial B_s}\rho|\nabla w|^2(y\cdot\nu)dS-\frac1{6}\int_{\partial B_s}\rho w^2(y\cdot\nu)dS-\l\int_{\partial B_s}\frac\rho w(y\cdot\nu)dS\\\notag
	=&-\int_{B_s}\nabla(\rho\cdot\nabla w)w_sdy-\frac13\int_{B_s}\rho ww_sdy+\l\int_{B_s}\frac\rho{w^2}w_sdy\\\notag
		&+\int_{\partial B_s}\rho(\nabla w\cdot\nu)w_sdS+\frac1{2}\int_{\partial B_s}\rho|\nabla w|^2(y\cdot\nu)dS\\\notag
		&-\frac1{6}\int_{\partial B_s}\rho w^2(y\cdot\nu)dS-\l \int_{\partial B_s}\frac\rho w(y\cdot\nu)dS\\
	=&-\int_{B_s}\rho|w_s|^2dy-\l\d e^{\frac s3}\int_{B_s}\rho\frac{|\nabla w|^2}{w^2}w_sdy+G(s),
\end{align}
where 
\begin{align*}
	G(s):=&\int_{\partial B_s}\rho(\nabla w\cdot\nu)w_sdS+\frac12\int_{\partial B_s}\rho|\nabla w|^2(y\cdot\nu)dS\\
		&-\frac16\int_{\partial B_s}\rho w^2(y\cdot\nu)dS-\l \int_{\partial B_s}\frac\rho w(y\cdot\nu)dS,
\end{align*}
$\nu$ is the exterior unit normal vector to $\partial\O$ and $dS$ is the surface area element. The first equality in \eqref{dE/ds} is followed by Lemma 2.3 \cite{Liu}. Let us estimate $G(s)$ as in Lemma 2.10 \cite{Guo}:
\begin{align}\label{estimate of G}\notag
	G(s)\leq&\int_{\partial B_s}\rho(\nabla w\cdot\nu)w_sdS+\frac12\int_{\partial B_s}\rho|\nabla w|^2(y\cdot\nu)dS\\
		\leq&C_1s^ne^{-\frac{s^2}4}+C_2s^{n-1}e^{-\frac{s^2}4}\lesssim s^ne^{-\frac{s^2}4},
\end{align}
since 
\begin{align}\label{estimate of w_s}
	|w_s|\leq C(1+|y|)+\frac w3\leq\tilde{C}(1+s),
\end{align}
due to Lemma \ref{lemma-w_infty} and the fact that $a$ is the quenching point. Hence, by integrating \eqref{dE/ds} in time from $a$ to $b$, we have that 
\begin{align}\label{energy estimate in both space and time}\notag
	\int_a^b\int_{B_s}\rho|w_s|^2dyds\leq& E[w](a)-E[w](b)\\
		&+C\int_a^b se^{\frac s3}\int_{B_s}\rho|\nabla w|^2dyds+\tilde{C}\int_a^bG(s)ds
\end{align}
for any $a<b$. Now we shall show that $w_\infty$ is independent of $s$. Let $a=m+s_j$, $b=m+s_{j+1}$ and $w=w_j$ in \eqref{energy estimate in both space and time}:
\begin{align}\label{integration on m+s_j to m+s_j+1}\notag
	&\int_m^{m+s_{j+1}-s_j}\int_{B_{s+s_j}}\rho\left|\frac{\partial w_j}{\partial s}\right|^2dyds\\
		\leq&E[w_j](m)-E[w_{j+1}](m)+C\int_{m+s_j}^{m+s_{j+1}}se^{\frac s3}\rho|\nabla w|^2dyds+\tilde{C}\int_{m+s_j}^{m+s_{j+1}}G(s)ds
\end{align}
for any integer $m$. Since $s_j+m\rightarrow \infty$ as $j\rightarrow\infty$, the third and the last term on the right-hand side of \eqref{integration on m+s_j to m+s_j+1} tend to zero, due to \eqref{condition} and \eqref{estimate of G}, respectively. Since $\nabla w_j(y,m)$ is bounded and indepdent of $j$, and $\nabla w_j(y,m)\rightarrow\nabla w_\infty(y,s)$ a.e. as $j\rightarrow\infty$, we have
\begin{align}
	\lim_{j\rightarrow\infty}E[w_j](m)=\lim_{j\rightarrow\infty}E[w_{j+1}](m):=E[w_\infty],
\end{align}
according to the dominated convergence theorem. Thus, the right-hand side of \eqref{integration on m+s_j to m+s_j+1} tends to zero as $j\rightarrow\infty$. Therefore
\begin{align}\label{integral}
	\lim_{j\rightarrow\infty}\int_m^M\int_{B_{s+s_j}}\rho\left|\frac{\partial w_j}{\partial s}\right|^2dyds=0
\end{align}
for each pair of $m$ and $M$. Now, from \eqref{estimate of w_s} where $\tilde{C}$ is independent of $j$, we get $\frac{\partial w_j}{\partial s}$ converges weakly to $\frac{\partial w_\infty}{\partial s}$. Since $\rho$ decreases expeonentially as $|y|\rightarrow\infty$ the integral in  \eqref{integral} is lower-semicontinuous, and we conclude that 
\[
	\int_m^M\int_{\mathbb{R}^n}\left|\frac{\partial w_\infty}{\partial s}\right|^2dyds=0.
\]
Since $m$ and $M$ are arbitrary, we show that $w_\infty$ is indepedent of $s$.

Since $\left|\frac{\partial w_j}{\partial s}\right|$ and $\nabla w_j$ are locally bounded in $\mathbb{R}^n\times(s_0,\infty)$ for some $s_0\gg1$, by Corollary \ref{coro-estimate on w}, $w_\infty$ is locally Lipschitzian. Each $w_j$ solves \eqref{w_a-eqn} and condition \eqref{condition} forces $e^s|\nabla w|^2\rightarrow0$, as $s\rightarrow+\infty$, so $w_\infty$ is a stationary weak solution to \eqref{stationary w-eqn}. Schauder's estimates (cf. \cite{F}) yields the desired regularity of $w_\infty$, i.e. $w_\infty$ is acctually a strong solution. 
\end{proof}

The solution to \eqref{stationary w-eqn} in one dimension has been investigated in \cite{FH}. And \cite{GuoCJM} studied the radially symmetric solution to this equation of dimension $n\geq2$. Combining Proposition \ref{prop-asymptotics} and their results, we assert that 
\begin{theorem}\label{thm-limit of w}
	Suppose $u$ is a solution to \eqref{F} quenching at $x=a$ in finite time $T$. Assume further that condition \eqref{condition} is satisfied. Then we have
	\[
		\lim_{t\rightarrow T^-}(1-u(x,t))(T-t)^{-\frac13}=(3\l)^{\frac13}
	\] 
uniformly on $|x-a|\leq C\sqrt{T-t}$ for any bounded constant $C$.
\end{theorem}
\begin{proof}
	It is shown in Theorem 2.1, \cite{FH} and Theorem 1.6, \cite{GuoCJM} that every non-constant (radially symmetric in dimension $n\geq2$) solution $w(y)$  to \eqref{stationary w-eqn} in $\mathbb{R}^n$ must be strictly increasing for sufficiently large $|y|$, and $w(y)\rightarrow\infty$, as $|y|\rightarrow\infty$. Therefore, $w_\infty$ has to be a constant solution, i.e. $w_\infty\equiv(3\l)^{\frac13}$.
\end{proof}

\subsection{Local expansion near the singularity}

In this subsection, we shall construct the local expansion of the solution $u=u(x,t)$ near the quenching point and the quenching time, provided $\O\in\mathbb{R}^n$ is a radially symmetric domain. It has been shown in Theorem \ref{quench at origin} that the origin is the only quenching point. Let us make the following nonlinear transformation as motivated by \cite{KL} and \cite{GPW}:
\begin{equation}\label{eqn-nonlinear mapping}
	\z = \frac1{3\l}(1-u)^3.
\end{equation}
Notice that $u=1$ maps to $\z=0$. In terms of $\z$, \eqref{F} transforms to
\begin{equation}\label{eqn-z}
	\left\{\begin{aligned}
		\z_t=&\triangle \z-\frac23\frac{|\nabla\z|^2}{\z}-\frac{\d\l^{\frac23}}{3^{\frac43}}\frac{|\nabla\z|^2}{\z^{\frac43}}-1,\quad(x,t)\in\O_T,\\
		\z(x,t)=&\frac1{3\l},\quad (x,t)\in\partial\O_T,\\
		\z(x,0)=&\frac1{3\l},\quad x\in\O.
	\end{aligned}\right.
\end{equation}
We shall find a formal power series solution to \eqref{eqn-z} near $\z=0$. As in \cite{KL} and \cite{GPW} we look for a locally radially symmetric solution to \eqref{eqn-z} in the form
\begin{equation}\label{eqn-z.expansion}
	\z(r,t)=\z_0(t)+\frac{r^2}{2!}\z_2(t)+\frac{r^4}{4!}\z_4(t)+\cdots,
\end{equation}
where $r=|x|$. Substituting \eqref{eqn-z.expansion} into \eqref{eqn-z} and collecting the coefficients in $r$, we obtain the following coupled ODEs for $\z_0$ and $\z_2$:
\begin{align}\label{eqn-z0z2}
	\z_0'=-1+n\z_2,\quad\z_2'=\frac{n+2}3\z_4-\frac43\frac{\z_2^2}{\z_0}-\frac{2\d\l^{\frac23}}{3^{\frac43}}\frac{\z_2^2}{\z_0^{\frac43}}.
\end{align}
We are interested in the solution with $\z_0(T)=0$, $\z_0'<0$ and $\z_2<0$ for $T-t>0$ and $T-t\ll1$. We shall assume that $\z_4\ll\frac{\z_2^2}{\z_0^{\frac43}}$ near the singularity. And it is clear that $\frac{\z_2^2}{\z_0}\ll\frac{\z_2^2}{\z_0^{\frac43}}$, since $\z_0\ll1$. Hence, \eqref{eqn-z0z2} reduces to 
\begin{align}\label{eqn-z0z2.1}
	\z_0'=-1+n\z_2,\quad\z_2'=-\frac{2\d\l^{\frac23}}{3^{\frac43}}\frac{\z_2^2}{\z_0^{\frac43}}.
\end{align}
Now we solve the system \eqref{eqn-z0z2.1} asymptotically as $t\rightarrow T^-$. We first assume that $n\z_2\ll1$ near $T$. This leads to $\z_0\sim T-t$ and the following differential equation for $\z_2$:
\begin{equation}\label{eqn-z2}
	\z_2'\sim-\frac{2\d\l^{\frac23}}{3^{\frac43}}\frac{\z_2^2}{(T-t)^{\frac43}}.
\end{equation}
By integrating \eqref{eqn-z2}, we obtain that
\begin{equation}\label{eqn-z2.1}
	\z_2\sim\frac{3^{\frac13}}{2\d\l^{\frac23}}(T-t)^{\frac13}+A\frac{(T-t)^{\frac13}}{\log{(T-t)}}+\cdots,
\end{equation}
for some unknown constant $A$. From \eqref{eqn-z2.1}, we observe that the consistency condition that $n\z_2\ll1$ as $t\rightarrow T^-$ is indeed satisfied. Substitute \eqref{eqn-z2.1} into \eqref{eqn-z0z2.1} for $\z_0$, we obtain for $t\rightarrow T^-$ that
\begin{equation}\label{eqn-z0}
	\z_0'\sim-1+n\left(\frac{3^{\frac13}}{2\d\l^{\frac23}}(T-t)^{\frac13}+A\frac{(T-t)^{\frac13}}{\log{(T-t)}}+\cdots\right).
\end{equation}
Using the method of dominant balance, we look for a solution to \eqref{eqn-z0} as $t\rightarrow T^-$ in the form
\[
	\z_0\sim(T-t)+(T-t)\left(B_0(T-t)^{\frac13}+B_1\frac{(T-t)^{\frac13}}{\log{(T-t)}}+\cdots\right),
\]
for some constants $B_0$ and $B_1$. A simple calculation yields that
\begin{equation}\label{eqn-z0.1}
	\z_0\sim(T-t)+(T-t)\left[-\frac{3^{\frac43}n}{8\d\l^{\frac23}}(T-t)^{\frac13}-\frac34nA\frac{(T-t)^{\frac13}}{\log{(T-t)}}+\cdots\right],\quad\textup{as}\ t\rightarrow T^-.
\end{equation}
The local form for $\z$ near quenching point is $\z\sim\z_0+\frac{r^2}2\z_2$. Using the leading term in $\z_2$ from \eqref{eqn-z2.1} and the first two terms in $\z_0$ from \eqref{eqn-z0.1}, we obtain the local form 
\begin{equation}\label{eqn-z.asymptotic}
	\z\sim(T-t)\left[1-\frac{3^{\frac13}n}{8\d\l^{\frac23}}(T-t)^{\frac13}+\frac{3^{\frac13}}{4\d\l^{\frac23}}\frac{r^2}{(T-t)^{\frac23}}+\cdots\right],
\end{equation}
for $r\ll1$ and $T-t\ll1$. Finally, using the nonlinear mapping \eqref{eqn-nonlinear mapping} relating $u$ and $\z$, we conclude that
\begin{equation}\label{eqn-u.asymptotic}
	u\sim1-\left[3\l(T-t)\right]^{\frac13}\left(1-\frac{3^{\frac13}n}{8\d\l^{\frac23}}(T-t)^{\frac13}+\frac{3^{\frac13}}{4\d\l^{\frac23}}\frac{r^2}{(T-t)^{\frac23}}+\cdots\right)^{\frac13}.
\end{equation}

\section{Numerical simulations}

\setcounter{equation}{0}

\subsection{Numerical experiments on pull-in voltage and quenching time}

In section 3, we investigate the pull-in voltages $\l_\d^*$ and the finite quenching time $T$ of \eqref{F}. We shall verify our results in section 3 by numerically computing $\l_\d^*$ and $T$ for some choice of domain $\O$. Let us consider the following two choices of $\O$:
\begin{align*}
	\O:&\quad\left[-\frac12,\frac12\right]\quad\textup{(slab)},\\
	\O:&\quad|x|\leq1,\quad x\in\mathbb{R}^2\quad\textup{(unit disk)}.
\end{align*}
To otbain $\l_l$ and $\l_{u,1}$ in \eqref{eqn-lower bound for ld} and Proposition \ref{prop-lower bound for ld}, we numerically solve $-\triangle\xi=1$ in $\O$ with Dirichlet boundary condition, yielding that $||\xi||_\infty\approx0.125$ for the slab and $||\xi||_\infty\approx0.712$ for the unit disk in $\mathbb{R}^2$. The first eigen pairs $(\mu_0,\phi_0)$ of the operator $-\triangle$ with Dirichlet boundary condition in $\O$ and with the normalization $\int_\O\phi_0dx=1$ are explicitly given below
\begin{align}
	\mu_0=\pi^2,&\quad\phi_0=\frac\pi2\sin{\left[\pi\left(x+\frac12\right)\right]}\quad\textup{(slab)},\\
\mu_0=z_0^2\approx5.783,&\quad\phi_0=\frac{z_0}{J_1(z_0)}J_0(z_0(|x|))\quad\textup{(unit disk)},
\end{align}
where $J_0$ and $J_1$ are Bessel functions, and $z_0\approx2.4048$ is the first zero of $J_0(z)$.

We first compute the pull-in voltage $\l_\d^*$ for various $\d$ in Table \ref{table-bounds for the pull-in voltages} for both slab (left column) and unit disk (right column). We use {\it bvp4c} in MatLab to determine $\l_\d^*$ (cf. \cite{SKR}). It is shown that $\l_\d^*$ decreases as $\d$ increases. And $\l_{u,1}$ for the case $\d=0.7$ and $\O$ is the slab provides a better upper bounds than the natural bound $\l_0^*$ (given by the comparison principle, as in \cite{Wq}).
\begin{table}
	\begin{minipage}{.4\linewidth}
	\centering
        \begin{tabular}{|c|c|c|c|}
\hline
		$\d$&$\l_\d^*$&$\l_l$&$\l_{u,1}$\\
\hline
		$0$&$1.440$&$1.1852$&$1.4622$\\
\hline
		$0.1$&$1.391$&$0.9581$&$1.4578$\\
\hline
		$0.7$&$1.196$&$0.4457$&$1.4314$\\
\hline
	\end{tabular}
\subcaption{the slab}
	\end{minipage}	
      \begin{minipage}{.4\linewidth}
	\centering
	  \begin{tabular}{|c|c|c|c|}
\hline
	$\d$&$\l_\d^*$&$\l_l$&$\l_{u,1}$\\
\hline
	$0$&$0.8030$&$0.2080$&$1.4622$\\
\hline
	$0.1$&$0.7890$&$0.2065$&$0.8523$\\
\hline
	$0.7$&$0.712$&$0.1979$&$0.8255$\\
\hline
        \end{tabular}
	\subcaption{the unit disk}
\end{minipage}
   	\caption{Pull-in voltages $\l_\d^*$ of \eqref{F} with $\d=0$, $0.1$ and $0.7$ for both the slab and the unit disk. The lower bound $\l_l$ in \eqref{eqn-lower bound for ld} and the upper bound $\l_{u,1}$ in Proposition \ref{prop-lower bound for ld} are also shown. {\sl Left:} slab; {\sl Right:} unit disk.}\label{table-bounds for the pull-in voltages}
\end{table}

Next, we verify the result in Proposition \ref{prop-delta large} by numerically computing the pull-in voltage for various $\d=0$, $0.7$, $7$, $70$, $700$ and $7000$. The pull-in voltage $\l_\d^*$ is also located by {\it bvp4c} in MatLab. It is clearly verified in Table \ref{table-pull-in voltage tends to zero} that $\l_\d^*\rightarrow0$, as $\d\rightarrow\infty$, for both the slab and the unit disk.

\begin{table}
	\begin{tabular}{|c|c|c|c|c|c|c|}
\hline
		$\d$&$0$&$0.7$&$7$&$70$&$700$&$7000$\\
\hline
		$\l_\d^*$ (slab) &$1.440$&$1.196$&$0.706$&$0.301$&$0.109$&$0.036$\\
\hline
		$\l_\d^*$ (unit disk) &$0.8030$&$0.712$&$0.472$&$0.218$&$0.081$&$0.028$\\
\hline
	\end{tabular}\\[5pt]
	\caption{The pull-in voltages $\l_\d^*$ tend to zero, as $\d\rightarrow\infty$ for both the slab and the unit disk.} \label{table-pull-in voltage tends to zero}
\end{table}

About the quenching time, we use the finite-difference scheme to compute the numerical soltuions to the nonlinear transformed equation \eqref{eqn-z}. The  detailed schemes are provided in section 6.2 below. We numerically verify in Table \ref{table-quenching time} that 
\begin{align}\label{eqn-limit l T}
	\lim_{\l\rightarrow\infty}\l T=\frac13
\end{align}
for the case without the fringing term. It is also shown numerically that \eqref{eqn-limit l T} no longer holds, for $\d>0$. Proposition \ref{prop-different lambda} has also been verified by various $\d$ and domains in Table \ref{table-quenching time}. Moreover, we observe from the results that $\lim_{\l\rightarrow\infty}\l T=0$ and the rate of convergence is independent of the fringing tem $\d$.
\begin{table}
	\begin{tabular}{|c|c|c|c|c|c|}
\hline
		$\d$&$\l$&$T_{slab}$&$\l T_{slab}$&$T_{disk}$&$\l T_{disk}$\\
\hline
		$0$&\makecell{$1.5$\\$10$\\$50$\\$100$}
&\makecell{$1.073664$\\$0.034122$\\$0.0066666$\\$0.003333$}
&\makecell{$1.610496$\\$0.34122$\\$0.3333$\\$0.3333$}
&\makecell{$0.292764$\\$0.033348$\\$0.006666$\\$0.00333$}
&\makecell{$ 0.439146$\\$0.33348$\\$0.333$\\$0.333$}\\
\hline
		$0.1$&\makecell{$2$\\$20$\\$200$\\$2000$}
&\makecell{$0.30837$\\$0.016692$\\$0.000816$\\$0.000048$}
&\makecell{$0.6167$\\$0.3338$\\$0.1632$\\$0.0960$}
&\makecell{$0.19011$\\$0.016668$\\$ 0.000816$\\$0.000048$}
&\makecell{$0.38022$\\$0.033336$\\$0.1632$\\$0.0960$}\\
\hline
		$1$&\makecell{$2$\\$20$\\$200$}
&\makecell{$0.24009$\\$0.008658$\\$0.000198$}
&\makecell{$0.4802$\\$0.1732$\\$0.0396$}
&\makecell{$0.18327$\\$0.008778$\\$0.000198$}
&\makecell{$0.36654$\\$0.17556$\\$0.0396$}\\
\hline
		$10$&\makecell{$2$\\$20$\\$200$}
&\makecell{$0.098892$\\$0.001392$\\$0.000066$}
&\makecell{$0.1978$\\$0.02784$\\$0.0132$}
&\makecell{$0.101538$\\$0.001398$\\$0.000066$}
&\makecell{$0.203076$\\$0.02796$\\$0.0132$}\\
\hline	
	\end{tabular}
	\caption{The quenching time $T_{slab}$ and $T_{disk}$ for $\d=0$, $0.1$, $1$ and $10$ with various $\l$ have been numerically computed, where $T_{slab}$ and $T_{disk}$ represent the quenching time for the slab $\left[-\frac12,\frac12\right]$ and the unit disk $|x|\leq 1$ in $\mathbb{R}^2$.}\label{table-quenching time}
\end{table}

\subsection{Numerical solution to \eqref{F}}

To numerically solve \eqref{F}, as suggested in \cite{GPW}, the tranformed problem \eqref{eqn-z} is more suitable for implementation. In fact, if we use the local behavior 
\[
	\z\sim(T-t)+\frac{3^{\frac13}}{4\d\l^{\frac23}}(T-t)^{\frac13}r^2,
\]
we get that
\[
	\frac{|\nabla\z|^2}{\z^{\frac43}}\sim\frac{3^{\frac23}(T-t)^{-\frac23}}{4\d^2\l^{\frac43}\left[r^{-\frac32}+\frac{3^{\frac13}}{4\d\l^{\frac23}}(T-t)^{-\frac23}r^{\frac12}\right]^{\frac43}},
\quad	\frac{|\nabla\z|^2}\z\sim\frac{\frac{3^{\frac23}}{4\d^2\l^{\frac43}}(T-t)^{\frac23}}{\frac{(T-t)}{r^2}+\frac{3^{\frac13}}{4\d\l^{\frac23}}(T-t)^{\frac13}}.
\]
Hence, the two terms $\frac{|\nabla\z|^2}{\z^{\frac43}}$ and $\frac{|\nabla\z|^2}{\z}$ in \eqref{eqn-z} is bounded in $r$, for any fixed $t$, even when $t$ is close to $T$. This allows us to use a simple finite-difference scheme to compute the numerical solutions to \eqref{eqn-z}. 

{\sl Experiment 1.} Let us first consider the domain slab $[-\frac12,\frac12]$ in one dimension with $\l=1$, $1.35$ or $3$ and $\d=0$ or $0.7$. This interval is discretized into $N+1$ pieces with $N=200$, i.e., $h=\frac1{N+1}\approx4.97512\times10^{-3}$ is the spartial mesh size. And the time step is labelled as $dt=6\times10^{-6}$. $\z_j^m$, for $j=1,\cdots,N+2$, is defined to be the discrete approximation to $\z\left(m\,dt,-\frac12+(j-1)h\right)$. The second-order accurate in space and first-order accurate in time scheme of \eqref{eqn-z} is 
\begin{equation}\label{eqn-slab scheme}
	\z_j^{m+1}=\z_j^m+dt\left(\frac{\z_{j+1}^m-2\z_j^m+\z_{j-1}^m}{h^2}-\frac{\left(\z_{j+1}^m-\z_{j-1}^m\right)^2}{6\z_j^mh^2}-\frac{\d\l^{\frac23}}{3^{\frac43}}\frac{\left(\z_{j+1}^m-\z_{j-1}^m\right)^2}{4\left(\z_j^m\right)^{\frac43}h^2}\right),
\end{equation}
$j=2,\cdots,N+1$, with $\z_1^m=\z_{N+2}^m=\frac1{3\l}$ for $m>0$ and $\z_j^0=\frac1{3\l}$ for $j=1,\cdots,N+2$. The time-step $dt$ is chosen to satisfy $dt<\frac{h^2}4$ for the stability of the discrete scheme. The experimental stop time is $T_{ex}=m\times dt$, where the $m$ is such that $\displaystyle\min_{j=1,\cdots,N+2}(\z_j^m-0)<10^{-10}$ for finite time quenching solution or $\displaystyle\max_{j=1,\cdots,N+2}(\z_j^{m+1}-\z_j^m)<10^{-10}$ for the globally existing solution. 
\begin{figure}
	\begin{minipage}{1.1\linewidth}
	\centering\includegraphics[trim = 15mm 85mm 15mm 85mm, clip, height=6cm, width= 11cm]{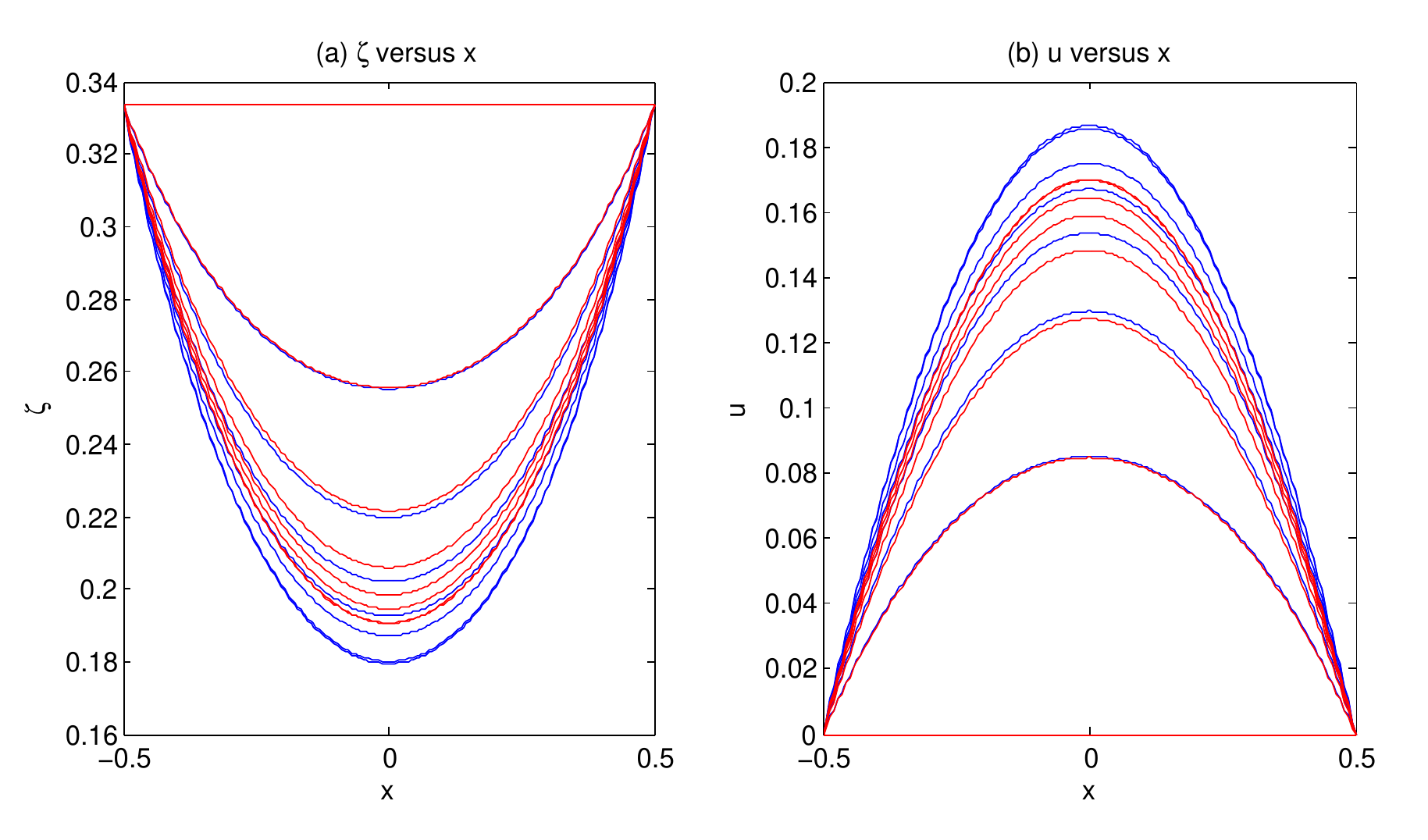}
		\subcaption{$\l=1$. We plot at times $t=0,0.1,0.2,0.3,0.4,0.5,2.0$ and the experimental stop time. Both solutions to \eqref{F} and \eqref{eqn-without fringing} increase towards a steady-state solution as $t$ increases.}\label{fig:subfig-lambda1}
	\end{minipage}
	\begin{minipage}{1.1\linewidth}
	\centering\includegraphics[trim = 15mm 80mm 15mm 80mm, clip, height=6cm, width= 11cm]{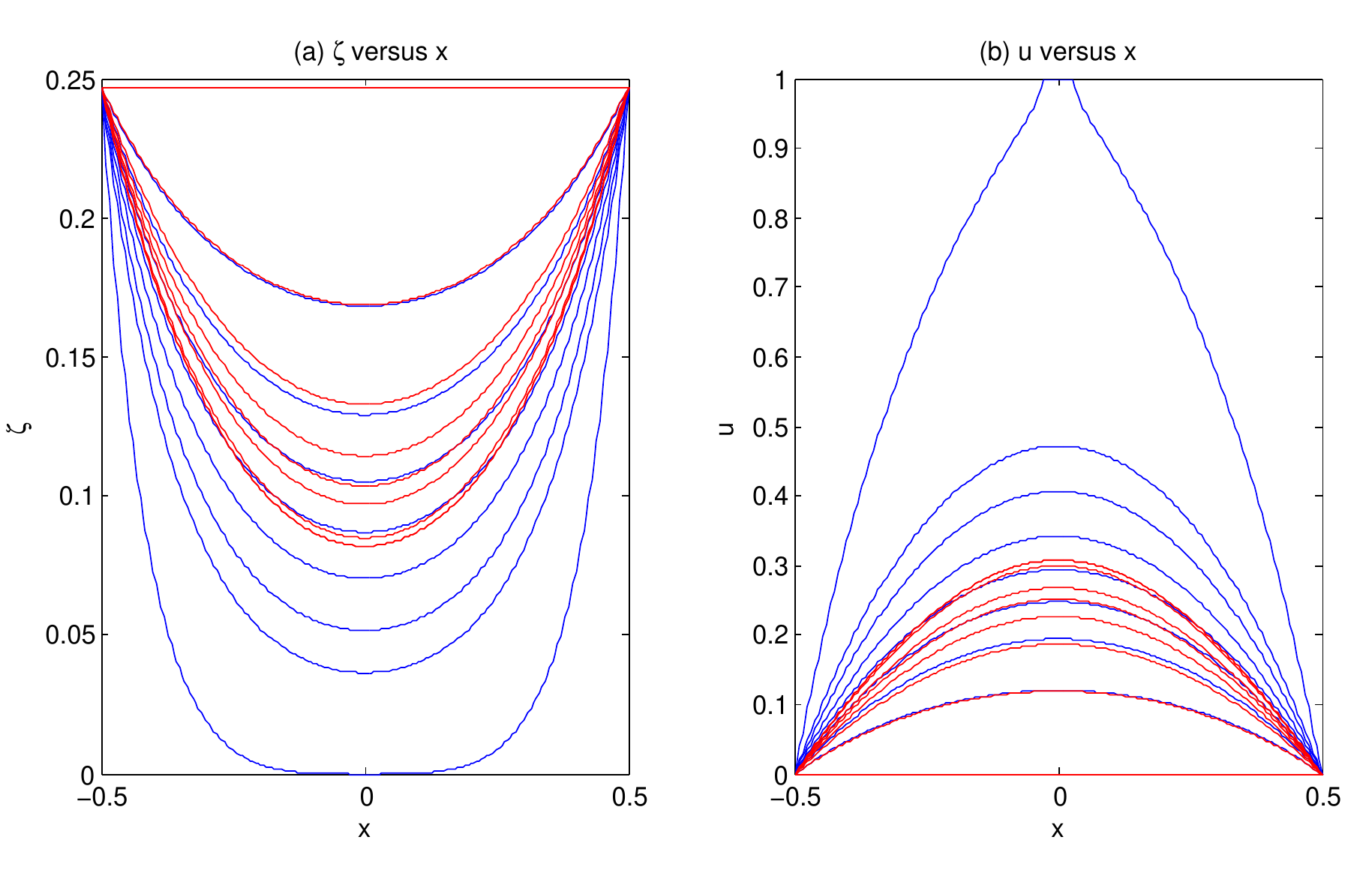}
		\subcaption{$\l=1.35$.  We plot at times $t=0,0.1,0.2,0.3,0.4,0.5,1.0,3.0$ and the experimental stop time for $\d=0$; while at times $t=0,0.1,0.2.0.3,0.4,0.5,0.6,0.66$ and the experimental stop time for $\d=0.7$. The solution to \eqref{eqn-without fringing} still globally exists; while that of \eqref{F} quenches in finite time.}\label{fig:subfig-lambda135}
	\end{minipage}
	\begin{minipage}{1.1\linewidth}
	\centering\includegraphics[trim = 10mm 85mm 10mm 85mm, clip, height=6cm, width= 11cm]{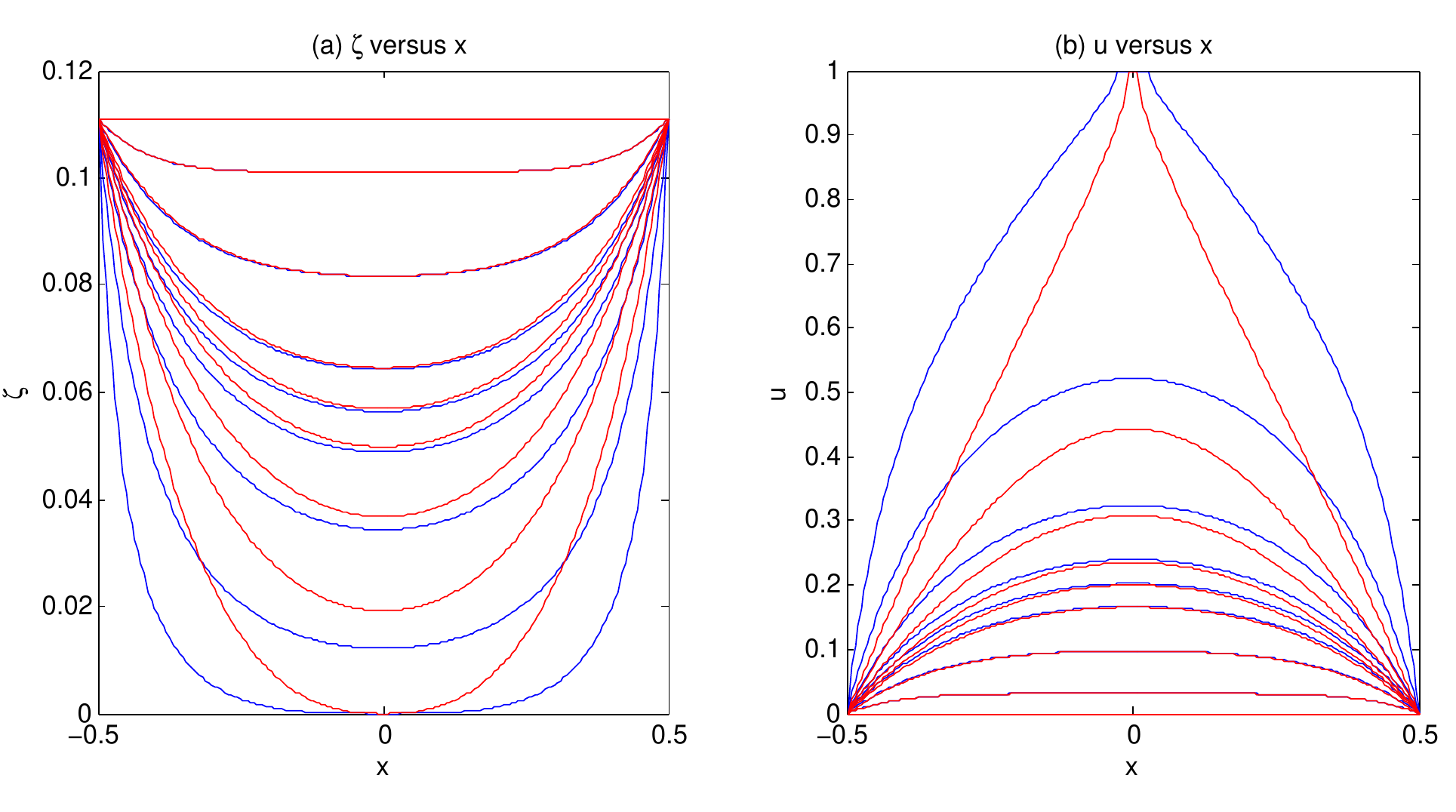}
	\subcaption{$\l=3$. We plot at times $t=0,0.01,0.03,0.05,0.06,0.07,0.09,0.12$ and the experimental stop time. Both solutions to \eqref{eqn-without fringing} and \eqref{F} quench in finite time.}\label{fig:subfig-lambda3}
	\end{minipage}
	\caption{{\sl Experiment 1:} For the slab domain $\left[-\frac12,\frac12\right]$ with different $\l$. We plot $\z$ and $u$ versus $x$ at a sequential times from the finite difference scheme \eqref{eqn-slab scheme} with $N=200$ and $dt=0.6\times10^{-5}$ and $\d=0$ or $0.7$. {\sl Left}: $\z$ versus $x$; {\sl Right}: $u$ versus $x$; {\sl Blue}: solution of \eqref{F}; {\sl Red}: that of \eqref{eqn-without fringing}.}\label{fig-slab}
\end{figure}

In Fig. \ref{fig-slab}, we plot $\z$ v.s. $x$ (left) and $u$ v.s. $x$ (right) from the discrete approximation \eqref{eqn-slab scheme} at a series of times. The solution to \eqref{F} with $\d>0$ is drawn in blue; while that of ($F_{\l,0}$) (cf. \eqref{eqn-without fringing}) is in red. Three different voltages are chosen $\l=1$, $1.35$ and $3$. It is suggested by the numerical simulation that the pull-in voltage of \eqref{eqn-without fringing} should be $1.35<\l^*<3$; while that of ($F_{\l,0.7}$) is between $1$ and $1.35$. The estimate of $\l^*$ matches well with the results in Table \ref{table-bounds for the pull-in voltages}, where $\l_0^*=1.440$ and $\l_{0.7}^*=1.196$. As to the profiles of the solutions to \eqref{F} with $\d=0.7$ and $\d=0$, the behavior is similar, if they both globally exist, see Fig. \ref{fig:subfig-lambda1}; the quenching profile of ($F_{\l,0.7}$) is much flatter than that of \eqref{eqn-without fringing}, if they both quench in finite time, see Fig. \ref{fig:subfig-lambda3}. The quenching times $T$ for both $\d=0$ and $\d=0.7$ in Fig. \ref{fig:subfig-lambda3} are numerically obtained to be around $0.1515$ and $0.134262$, respectively. This numerically verifies Remark \ref{remark-different delta}.

{\sl Experiment 2.} When we consider the unit disk in two dimension, a second-order accurate in space and first-order accurate in time discrete approximation for \eqref{eqn-z}, with spartial mesh size $h$, on $0\leq r\leq 1$ and $t\geq0$ is
\begin{align}\label{eqn-disk scheme}\notag
	\z_j^{m+1}=&\z_j^m+dt\left(\frac{\z_{j+1}^m-2\z_j^m+\z_{j-1}^m}{h^2}+\frac{\z_{j+1}^m-\z_{j-1}^m}{2hr_j}\right.\\
		&\phantom{\z_j^m+dtaa}\left.-\frac{(\z_{j+1}^m-\z_{j-1}^m)^2}{6\z_j^mh^2}-\frac{\d\l^{\frac23}}{3^{\frac43}}\frac{(\z_{j+1}^m-\z_{j-1}^m)^2}{4\left(\z_j^m\right)^{\frac43}h^2}-1\right),
\end{align}
where $r_j=jh$. According to \cite{MM}, the discrete approximation for $\z_1$ at the origin $r=0$ is 
\[
	\z_1^{m+1}=\z_1^m+\frac{4dt}{h^2}(\z_2^m-\z_1^m).
\]
The condition at $r=1$ is $\z_{N+2}^m=\frac1{3\l}$, and the initial condition is $\z_j^0=\frac1{3\l}$, for $j=1,\cdots,N+2$.  The experimental stop time is $T_{ex}=m\times dt$, where the $m$ is such that $\displaystyle\min_{j=1,\cdots,N+2}(\z_j^m-0)<10^{-10}$ for finite time quenching solution or $\displaystyle\max_{j=1,\cdots,N+2}(\z_j^{m+1}-\z_j^m)<10^{-10}$ for the globally existing solution. 

In Fig. \ref{fig-diskdelta07lambda1}, we plot $\z$ v.s. $|x|$ (left) and $u$ v.s. $|x|$ (right) from the discrete approximation \eqref{eqn-disk scheme} with the voltage chosen to be $\l=1$ at times $t=0.1,0.2,0.3,0.4,0.5$ and the experimental stop time $T_{ex}$. The solution to \eqref{F} with $\d>0$ is drawn in blue; while that of ($F_{\l,0}$) or \eqref{eqn-without fringing} is in red. It is suggested by the numerical simulations that both the pull-in voltage $\l^*$ of ($F_{\l,0}$) and that $\l_{0.7}^*$ of ($F_{\l,0.7}$) are less than $1$. This coincides with $\l^*=0.8030$ and $\l_{0.7}^*=0.712$ in Table \ref{table-bounds for the pull-in voltages} or Table \ref{table-pull-in voltage tends to zero}. And the quenching times $T$ with $\d=0$ and $0.7$ are numerically obtained to be around $0.7076$ and $0.578232$, respectively. 

\begin{figure}[]
	\includegraphics[trim = 10mm 85mm 10mm 85mm, clip, height=6cm, width= 13cm]{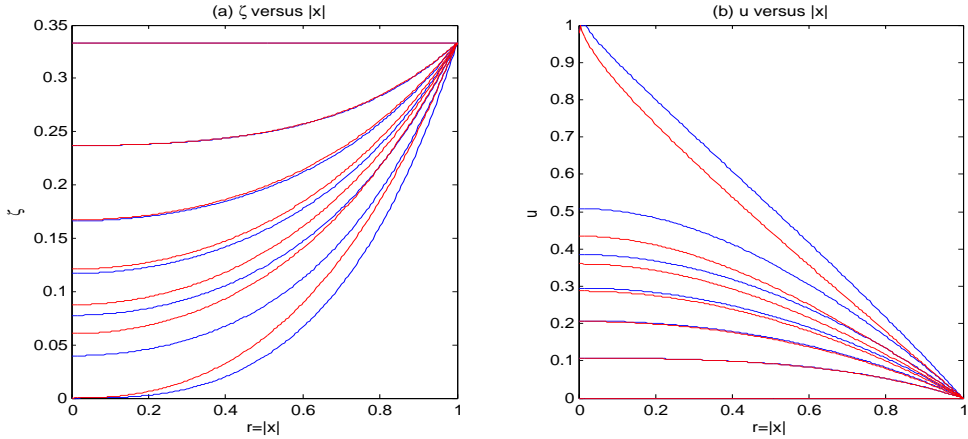}
	\caption{{\small {\sl Experiment 2:} For the unit disk domain in two dimension with $\l=1$. We plot $\z$ and $u$ versus $x$ at times $t=0.1,0.2,0.3,0.4,0.5$ and the experimental stop time from the finite difference scheme \eqref{eqn-slab scheme} with $N=200$ and $dt=0.6\times10^{-5}$ and $\d=0$ or $0.7$. {\sl Left}: $\z$ versus $|x|$; {\sl Right}: $u$ versus $|x|$; {\sl Blue}: solution of \eqref{F}; {\sl Red}: that of \eqref{eqn-without fringing}.} }\label{fig-diskdelta07lambda1}
\end{figure}

{\sl Experiment 3.} Let us examine the local approximation constructed in \eqref{eqn-u.asymptotic} numerically. From Experiment 1, the numerically obtained the quenching time for ($F_{3,0.7}$) in the slab domain $\left[-\frac12,\frac12\right]$ is $0.134262$; and from Experiment 2, the quenching time for ($F_{1,0.7}$) in the unit disk of dimension two is around $0.578232$. In Fig. \ref{fig-localcompare}, we plot $\z$ v.s. $x$ and $|x|$ of the discrete approximation \eqref{eqn-slab scheme} with $\l=3$ and \eqref{eqn-disk scheme} with $\l=1$ at time $t=0.134004$ and $t=0.57822$, respectively, in blue. At the same time, we plot the local approximation obtained in \eqref{eqn-z.asymptotic} in black. From Fig. \ref{fig-localcompare}, the local approximation \eqref{eqn-z.asymptotic} matches the numerical solutions well.
\begin{figure}[]
	\includegraphics[trim = 20mm 90mm 20mm 90mm, clip, height=6cm, width= 13cm]{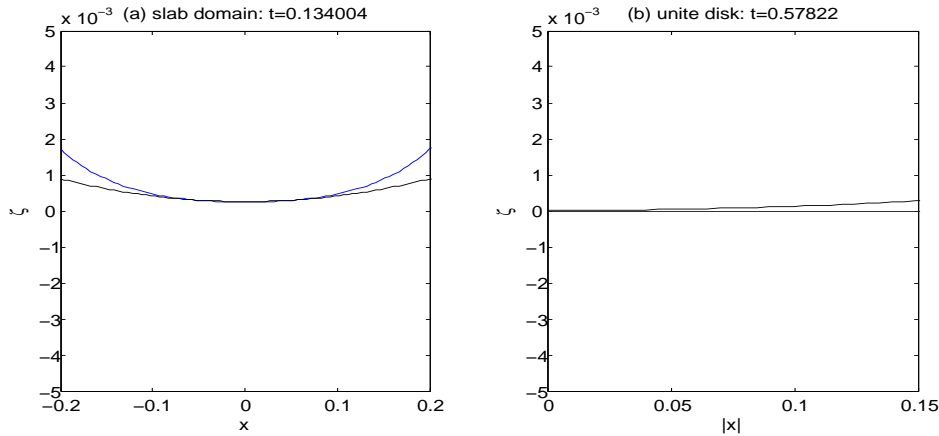}
	\caption{{\small {\sl Experiment 3:} We plot $\z$ versus $x$ or $|x|$ for (a) slab domain and (b) unit disk of the discrete approximations of \eqref{eqn-z}. {\sl Blue}: the numerical solution given by \eqref{eqn-slab scheme} (left column) or \eqref{eqn-disk scheme} (right column); {\sl Black}: the local approximations given by \eqref{eqn-z.asymptotic}.}}\label{fig-localcompare}
\end{figure}

\section{Conclusion}

\setcounter{equation}{0}

In this paper, we study the equation \eqref{F} modelling the MEMS device with the fringing term $\d>0$. We first show that the pull-in voltage $\l_\d^*>0$ obtained in \cite{WY} is the watershed of globally existing solution and the finite time quenching solution of \eqref{F}. To be more precisely, if $\l\leq\l_\d^*$, then the unique solution to \eqref{F} exists globally; otherwise, the solution will quench in finite time $T<\infty$. 

According to the comparision principle, a natural upper bound of $\l_\d^*$ is $\l^*$, the pull-in voltage of ($F_{\l,0}$). In this paper, it has been slightly improved in Proposition \ref{prop-lower bound for ld} for $\d\ll1$ and numerically verified in Table \ref{table-bounds for the pull-in voltages}. Moreover, we prove that $\lim_{\d\rightarrow\infty}\l_\d^*=0$. This has been validated numerically in Table \ref{table-pull-in voltage tends to zero}. 

About the quenching time $T$, for $\l>\l_\d^*$, we show that it satisfies $T\lesssim\frac1\l$, which differs from that corresponding to ($F_{\l,0}$) where $\lim_{\l\rightarrow\infty}\l T=\frac13$. We conjecture from Table \ref{table-quenching time} that $\lim_{\l\rightarrow\infty}\l T=0$ and the rate of convergence is independent of $\d$.

By adapting the moving-plane argument as in \cite{GNN}, we show that the quenching set of \eqref{F} is a compact set in $\O$, if $\O\subset\mathbb{R}^n$ is a bounded convex set. Furthermore, if $\O=B_R(0)$, the ball centered at the origin with the radius $R$, then the origin is the only quenching point. This is clearly seen from Fig. \ref{fig-slab} and Fig. \ref{fig-diskdelta07lambda1}.

Finally, we investigate the quenching behavior of the solution to \eqref{F} with $\l>\l_\d^*$. It is shown in this paper that, under certain condition, if $u$ is the solution to \eqref{F} quenching at $x=a$ in finite time $T$, then it satisfies \[
	\lim_{t\rightarrow T^-}(1-u(x,t))(T-t)^{-\frac13}=(3\l)^{\frac13}.
\] 
More refined asymptotic expansion is given in \eqref{eqn-u.asymptotic}. And it has been verified numerically in Fig. \ref{fig-localcompare} that this is a good local approximation.

\end{document}